\documentclass[12pt,reqno]{amsart}

\usepackage{graphics}
\usepackage{latexsym}
\usepackage[dvips]{epsfig}
\usepackage{color}
\usepackage{graphicx}
\usepackage{amsmath}
\usepackage{amssymb}
\usepackage{amsfonts}
\usepackage{eufrak}
\usepackage{amstext}
\usepackage{amsopn}
\usepackage{amsbsy}
\usepackage{amscd}
\usepackage{amsxtra}
\usepackage{amsthm}
\usepackage{bm}
\usepackage{CJK}

\newcommand{\R}{\mathbb{R}}

\newcommand{\vect}[2]{\ensuremath{\left( \begin{array}{c}
            #1 \\
            #2
            \end{array}
        \right)}}

\def\theequation{\thesection.\@arabic\c@equation}
\makeatother

\renewcommand{\theequation}{\thesection.\arabic{equation}}
\newtheorem{lemma}{Lemma}[section]

\newtheorem{proposition}{Proposition}[section]
\newtheorem{corollary}{Corollary}[section]

\newtheorem{theorem}{Theorem}[section]

\newcommand{\be}{\begin{equation}}
\newcommand{\ben}{\begin{equation*}}
\newcommand{\ee}{\end{equation}}
\newcommand{\een}{\end{equation*}}
\newcommand{\BL}{\begin{lemma}}
\newcommand{\EL}{\end{lemma}}
\newcommand{\BT}{\begin{theorem}}
\newcommand{\ET}{\end{theorem}}
\newcommand{\BP}{\begin{proposition}}
\newcommand{\EP}{\end{proposition}}
\newcommand{\BC}{\begin{corollary}}
\newcommand{\EC}{\end{corollary}}
\def\bs{\begin{split}}
\def\es{\end{split}}

\title[Infinitely many solutions]{Infinitely many positive solutions for nonlinear equations with non-symmetric potential}

\author{Weiwei Ao}
\address{Department of Mathematics, Chinese University of Hong Kong, Shatin, Hong Kong.  {\sl wwao@math.cuhk.edu.hk}}
\author{ Juncheng Wei}
\address{Department of Mathematics, Chinese University of Hong Kong, Shatin, Hong Kong,and
Department of Mathematics, University of British Columbia, Vancouver, B.C., Canada, V6T 1Z2.
{\sl wei@math.cuhk.edu.hk}}

\begin{document}

\maketitle
\begin{abstract}
We consider the following nonlinear Schrodinger equation
\[
\left\{\begin{array}{l}
\Delta u-(1+\delta V)u+f(u)=0 \ \ \mbox{ in }\R^N,\\
u>0 \ \mbox{in} \ \R^N, u\in H^1(\R^N)
\end{array}
\right.
\]
where $V$ is a potential satisfying some decay condition and $ f(u)$ is a superlinear nonlinearity satisfying some nondegeneracy condition. Using localized energy method, we prove that there exists some $\delta_0$ such that for $0<\delta<\delta_0$, the above problem  has infinitely many positive solutions. This  generalizes and gives a new proof of the results by Cerami-Passaseo-Solimini \cite{CPS}. The new techniques allow us to establish the existence of infinitely many positive bound states for elliptic systems.
\end{abstract}
\setcounter{equation}{0}

\section{Introduction}\label{1}
In this paper, we consider nonlinear Schrodinger equations and systems with non-symmetric potentials. We are interested in the multiplicity of positive solutions.
\subsection{Nonlinear Schrodinger equation with non-symmetric potential }\label{1.1}

We first consider the following equation:
\begin{equation}\label{p-1}
\left\{\begin{array}{l}
\Delta u-(1+\delta V(x))u+f(u)=0 \ \ \mbox{ in }\R^N\\
u>0 \ \mbox{in} \ \R^N,\ \ \ u\in H^1(\R^N)
\end{array}
\right.
\end{equation}
where $N\geq 2$ , $\delta$ is a positive constant and the potential $V$ is a continuous function satisfying suitable decay assumption, but without any symmetry. We are interested in the existence of {\em infinitely many} positive solutions of equation (\ref{p-1}).

Equation (\ref{p-1}) arises in the study of solitary waves in nonlinear equations of the Klein-Gordon or Schrodinger type and has been under extensive studies in recent years.

Consider the following problem first
\begin{equation}\label{p1}
\Delta u-V(x)u+f(u)=0, u>0 \ \ \mbox{ in }\R^N, \ \ u\in H^1(\R^N).
\end{equation}
If $f$ satisfies global Ambrosetti-Rabinowtiz condition and $V$ satisfies
\begin{equation}\label{p200}
\inf_{y \in \R^N} V(y) < \lim_{|x|\to\infty} V(x),
\end{equation}
then, using the concentration compactness principle \cite{L1,L2},
one can show that \eqref{p1}  has a
least energy solution. See for example \cite{DN,L1,L2,R}.

 But if \eqref{p200} does not hold, \eqref{p1} may not have least energy solution. So, one needs
to find solution with higher energy level. For results in this direction, the readers can refer to \cite{BL,BLi}.

On the other hand, if we consider the following semi-classical
problem:
\begin{equation}\label{2-18-4}
\varepsilon^2\Delta u - V(y) u + u^p=0, \quad u>0, \quad \lim_{|y|\to+\infty}
u(y)=0,
\end{equation}
where $\varepsilon>0$ is a small parameter and $p$ subcritical, then the number of the
critical points of $V(y)$ (see for example
\cite{ABC, O},\cite{DF1}--\cite{DF4},\cite{FW,O,W}), the type
of the critical points of $V(y)$ (see for example
\cite{BT1,KW,NY}, and the topology of the level set of $V(y)$
 \cite{AMN1, AMN2,BT2, FM}, can affect the number of the
solutions for \eqref{2-18-4}. The construction of single and multiple spikes in degenerate setting is done by Byeon-Tanaka \cite{BT1, BT2}.  In particular, we mention the following multiplicity result due to Kang-Wei \cite{KW} (see \cite{BT1} for general $f(u)$): If $V(x)$ has a local maximum point, then for any fixed integer $K$, there exists $ \epsilon_K>0$ such that for $\epsilon <\epsilon_K$ there are solutions with $K$ spikes. So for the
singularly perturbed problems \eqref{2-18-4}, the
parameter $\varepsilon$ will tend to zero as the number of the
solutions tends to infinity. Thus all these results do not give
multiplicity result for  \eqref{p1}.

About the existence of infinitely many positive solutions, Coti-Zelati and Rabinowitz  \cite{CR1, CR2} first proved the existence of arbitrarily bumps (hence infinitely many solutions) for (\ref{p1}) when $V$ is a periodic function in $\R^N$. (See Sere \cite{S} for related work on Hamiltonian systems.) As far as we know, without periodicity nor smallness of the parameters, the first result on the existence of infinitely many positive solutions was due to Wei-Yan \cite{WY}. (Another variational proof was given in \cite{DS}.) They proved the existence of infinitely many non radial positive bump solutions for (\ref{p1}) under the following assumption at infinity
$$ V(x)=V(|x|)= V_{\infty}+\frac{a}{|x|^m} + O(\frac{1}{|x|^{m+\sigma}}).$$

In a recent remarkable paper \cite{CPS},  Cerami-Passasseo-Solimini developed a localized Nehari's manifold argument  and  localized variational method to prove the existence of infinitely many positive solutions of the following equation
\begin{equation}\label{cps}
\left\{\begin{array}{l}
\Delta u-(1+\delta V)u+u^p=0 \ \ \mbox{ in }\R^N\\
u>0 \ \mbox{in} \ \R^N, \ u\in H^1(\R^N)
\end{array}
\right.
\end{equation}
where the potential $V$ satisfies suitable decay assumption (see below (H1)-(H2)).

The purpose of the first part of this paper is  two folds. Firstly, we want to generalize the results of \cite{CPS} for more general nonlinearity, i.e, we consider a more general equation (\ref{p-1}).
 Secondly, we will give another proof of the results of \cite{CPS}, in the spirit of Liapunov-Schmidt reduction.

In Section \ref{2}, we  assume that $f: \R\to \R$ satisfies the following two conditions:
\begin{itemize}
\item[$(f_1)$]
$f: \R\to \R$ is of class $C^{1+\sigma}$ for some $0<\sigma\leq 1$ and $f(u)=0$ for $u\leq 0$.

\item[$(f_2)$]
The equation
\begin{equation}\label{nondeg}
\left\{\begin{array}{l}
\Delta w-w+f(w)=0, \ w>0 \mbox{ in }\ \R^N\\
w(0)=\max_{y\in \R^N}w(y), \ w\to 0 \mbox{ as } |y|\to \infty
\end{array}
\right.
\end{equation}
has a  nondegenerate solution $w$, i.e.,
\begin{equation}
\operatorname{Ker}(\Delta -1+f'(w))=\operatorname{Span}\{\frac{\partial w}{\partial y_1},\cdots, \frac{\partial w}{\partial y_N}\}.
\end{equation}

\end{itemize}

We note that the function
\begin{equation}
f(t)=t^p-at^q, \ \mbox{ for } t\geq 0
\end{equation}
with a constant $a\geq 0$ satisfies the above assumptions $(f_1)-(f_2)$ if $1<q<p<(\frac{N+2}{N-2})_+$. Nondegeneracy is a generic condition.  We should remark that there do exist nonlinearities with degenerate ground states; the first example seems to be given by Dancer \cite{D}. See also Polacik \cite{P}.

Under the above assumptions, we know that there exists a unique positive eigenvalue of the operator $\Delta-1+f'(w)$. That is, there exists an unique eigenvalue $\lambda_1>0$ and its corresponding eigenfunction $\Phi_0$ (which can be made positive and radially symmetric) satisfying
\begin{equation}
\Delta \Phi_0-\Phi_0+f'(w)\Phi_0=\lambda_1\Phi_0, \ \Phi_0\in H^1(\R^N).
\end{equation}
This function will play important role in our secondary Liapunov-Schmidt reduction (see Section \ref{2.2} below).

The energy functional associated with (\ref{p-1}) is
\begin{equation}
J(u)=\frac{1}{2}\int_{\R^N}|\nabla u|^2+(1+\delta V)u^2 dx-\int_{\R^N}F(u)dx,
\end{equation}
where $F(u)=\int_0^uf(s)ds$.

Let us now introduce the assumptions on $V(x)$ (similiar to \cite{CPS})
\begin{equation}
\left\{\begin{array}{l}
(H1)\ \ V(x)\to 0 \ \ \mbox{ as } |x|\to \infty,\\
(H2)\ \ \ \exists \ \ 0<\bar{\eta}<1, \lim_{|x|\to \infty}V(x)\ e^{\bar{\eta}|x|}=+\infty.\\
\end{array}
\right.
\end{equation}

We now state the main theorem in this paper:

\begin{theorem}\label{theo1}
Let $f$ satisfies assumption $(f_1)-(f_2)$, the potential $V$ satisfies assumption $(H1)-(H2)$. Then there exists a positive constant $\delta_0$, such that for $0<\delta<\delta_0$, problem (\ref{p-1}) has infinitely many positive solutions.
\end{theorem}

 In the following we sketch the main steps in the proof of Theorem \ref{theo1}.

\subsection{Sketch of the proof of Theorem \ref{theo1}}\label{1.3}

 We introduce some notations first. Let $\rho>0$ be a real number such that $w(x)\leq ce^{-|x|} $ for $|x|>\rho$ and some constant $c$ independent of $\rho$ large. Now we define the configuration space,
\begin{equation}
\label{lambda}
\Lambda_1=\R^N,\ \ \Lambda_k:= \{(Q_1,\cdots,Q_k)\in \R^N|\min_{i\neq j}|Q_i-Q_j|\geq \rho\} , \forall k>1.
\end{equation}

Let $w$ be the nondegenerate solution of (\ref{nondeg}) and $k \geq 1$ be an integer.  Define the sum of $k$ spikes as
\begin{equation}
w_{Q_i}=w(x-Q_i), \mbox{ and }\ \ w_{Q_1, \cdots, Q_k}=\sum_{i=1}^kw_{Q_i}.
\end{equation}

Let the opertaor be
\begin{equation}
S(u)=\Delta u-(1+\delta V)u+f(u).
\end{equation}

Fixing $\mathbf{Q}_k=(Q_1,\cdots, Q_k)\in \Lambda_k$, we define the following functions as the approximate kernels:
\begin{equation}\label{zij}
Z_{ij}=\frac{\partial w_{Q_i}}{\partial x_j}\chi_i(x), \mbox{ for } i=1,\cdots,k, \ j=1,\cdots,N,
\end{equation}
where $w_{Q_i}(x)=w(x-Q_i)$, $\chi_i(x)=\chi(\frac{2| x-Q_i|}{(\rho-1)})$ and $\chi(t)$ is a cut off function , such that $\chi(t)=1$ for $|t|\leq 1$ and $\chi(t)=0$ for $|t|\geq \frac{\rho^2}{\rho^2-1}$. Note that the support of $Z_{ij}$ belongs to $B_{\frac{\rho^2}{2(\rho+1)}}(Q_i)$.

Using  $w_{Q_1,\cdots, Q_k}$ as the approximate solution and performing  the Liapunov-Schmidt reduction, we can show that there exists a constant $\rho_0$, such that for $\rho\geq\rho_0$, and $\delta<c_\rho$, for some constant $c_\rho$ depend on $\rho$ but independent of $k$ and $\mathbf{Q}_k$, we can find a $\phi_{\mathbf{Q}_k}$ such that
\begin{equation}
S(w_{Q_1,\cdots,Q_k}+\phi_{\mathbf{Q}_k})=\sum_{i=1,\cdots,k,j=1,\cdots,N}c_{ij}Z_{ij},
\end{equation}
and we can show that $\phi_{\mathbf{Q}_k}$ is $C^1$ in $\mathbf{Q}_k$. This is done is Section \ref{2.1}.

After that, for any $k$, we define a new function
\begin{equation}
\mathcal{M}(\mathbf{Q}_k)=J(w_{Q_1,\cdots,Q_k}+\phi_{\mathbf{Q}_k}),
\end{equation}
we maximize $\mathcal{M}(\mathbf{Q}_k)$ over $\bar{\Lambda}_k$.

There are two main difficulties in the maximization process. First, we need to show that the maximum points will not go to infinity. This is guaranteed by the slow decay assumption on the potential $V$. Second, we have to detect the difference in the energy when the spikes move to the boundary of the configuration space.  In the second step, we use the induction method and detect the difference of the $k$-spikes energy and the $k+1$-spikes energy. A crucial estimate is Lemma \ref{lemma501}, in which we prove that the accumulated error can be controlled from step $k$ to step $k+1$. To prove this, we perform a secondary Liapunov-Schmidt reduction. This is done is Section \ref{2.2} and \ref{2.3}. Finally in Section \ref{2.4}, we give the proof of Theorem \ref{theo1}.

Unlike  the variational method and Nehari's manifold arguments in \cite{CPS}, our main idea is to use the Liapunov-Schmidt reduction method. The only assumption we need is the nondegeneracy of the bump. We have no requirements on the structure of the nonlinearity. We note that the nondegeneracy is also needed in arguments of \cite{CPS}.  Our approach is different. It handles more general nonlinearities and can be readily applied to other similar problems such as elliptic systems and magnetic Ginzburg-Landau equations (\cite{PTW}).

In the following we present the applications of our techniques to elliptic systems in which the bump can have higher Morse index.

 \subsection{Nonlinear Schrodinger system with non-symmetric potential}\label{1.2}

As we mentioned above, our approach can be applied to other problems such as elliptic system. So in this section, we apply our method to the elliptic system. We consider the following nonlinear Schrodinger system in $\R^N$ ($N\leq 3$)
\begin{equation}\label{p-2}
\left\{\begin{array}{l}
-\Delta u+(1+\delta a(x))u=\mu_1|u|^2u+\beta v^2u\\
-\Delta v+(1+\delta b(x))v=\mu_2|v|^2v+\beta u^2v
\end{array}
\right.
\end{equation}
where  $\delta$ is a constant and the potential $a(x),b(x)$ are continuous functions satisfying suitable decay assumption, but without any symmetry property.

This type of system arise when one considers the standing wave solutions of the time dependent $n-$coupled Schrodinger systems of the form with $n=2$
\begin{equation}\label{b}
\left\{\begin{array}{l}
-i\frac{\partial}{\partial t}\Phi_j=\Delta \Phi_j-V_j(x)\Phi_j+\mu_j|\Phi_j|^2\Phi_j+\Phi_j\sum_{l=1,k\neq j}^N\beta_{jk}|\Phi_l|^2, \mbox{ in } \ \R^N,\\
\Phi_j=\Phi_j(x,t)\in C, \ t>0,\ j=1,\cdots,n
\end{array}
\right.
\end{equation}
where $\mu_j$ and $\beta_{jl}=\beta_{lj}$ are constants.  The system (\ref{b}) arises in applications of many physical problems, especially in the study of incoherent solitons in nonlinear optics. Physically, the solution $\Phi_j$ denotes the $j-$th component of the beam in Kerr-like photorefractive media. The positive constant $\mu_j$ is for self-focusing in the $j-$th component of the beam. The coupling constant $\beta$ is the interaction between the first and the second component of the beam. As $\beta>0$, the interaction is attractive, while the interaction is repulsive if $\beta<0$.


Mathematical work on systems of nonlinear Schrodinger equations have been studied extensively in recent years, see for example \cite{BDW, CTV,LW1,NTTV, TV, WW1, WW2} and references therein. Phase separation has been proved in several cases with constant potentials such as in the work \cite{BDW, CTV, DWW, NTTV, WW1, WW2} as the coupling constant $\beta$ tends to negative infinity. In symmetric case ($a=b=0, \mu_1=\mu_2$), \cite{WW2} gives infinitely many non-radial positive solutions for $\beta\leq -1$ which are potentially segregated type. In a recent paper of Peng and Wang \cite{PW}, the authors considered the multiplicity of solutions . They proved the existence of infinitely many solutions of synchronized type to (\ref{p-2}) for radial symmetric potentials $a(|x|),b(|x|)$ satisfying some algebra decay assumption. Their proof is in the spirit of the work \cite{WY}.

The second result of this paper concerns the existence of infinitely many synchronized solutions for potentials without any symmetry assumption.

We assume that $a(x), b(x)$ satisfy the following conditions:
\begin{equation}
\left\{\begin{array}{l}
(H'_1)\ \ a(x), \ b(x)\to 0 \ \ \mbox{ as } |x|\to \infty,\ \ a(x), b(x)\geq 0 \ \ as\ \  |x|\to \infty,\\
(H'_2) \ \ \exists\  0<\bar{\eta}<1,\ \ \lim_{|x|\to \infty}(\alpha^2a(x)+\gamma^2b(x))e^{\bar{\eta}|x|}=+\infty,
\end{array}
\right.
\end{equation}
where $\alpha,\gamma$ are constants defined in (\ref{alpha}).

The energy functional associated with problem (\ref{p-2}) is
\begin{eqnarray*}
J(u,v)&=&\frac{1}{2}\int_{\R^N}|\nabla u|^2+(1+\delta a)u^2+|\nabla v|^2+(1+\delta b)v^2dx\\
&-&\frac{1}{4}\int_{\R^N}\mu_1 u^4+\mu_2 v^4dx-\frac{\beta}{2}\int_{\R^n}u^2v^2dx, \ \ u, v\in H^1(\R^N).
\end{eqnarray*}

The second result of this paper is as follows:

\begin{theorem}\label{theo2}
Let the potential $a,b$ satisfies assumption $(H'_1), (H'_2)$. Then there exists $\beta^*>0$, and  $\delta_0>0$, such that for $\beta\in (-\beta^*,0)\cup (0,\min\{\mu_1,\mu_2\})\cup (\max\{\mu_1,\mu_2\},\infty)$, and $0<\delta<\delta_0$, problem (\ref{p-2}) has infinitely many positive synchronized solutions.
\end{theorem}

The main technical difference between the scalar problem (\ref{p-1}) and the system (\ref{p-2}) is that the system has higher Morse index for the bump profile. Since we only the nondegeneracy of the bump, we can still perform the secondary Liapunov-Schmidt reduction.

\medskip

The rest of the paper is organized as follows: Theorem \ref{theo1} is proved in  Section \ref{2}. In Section \ref{3}, we give the proof of Theorem \ref{theo2}.

\medskip

Throughout this paper, unless otherwise stated, the letters $c,
C$ will always denote various generic constants that are independent of $k$ for $\delta$ small enough.

\bigskip

\noindent
{\bf Acknowledgment.}  Juncheng Wei was supported by a GRF grant from RGC of Hong Kong.

\section{Infinitely many solutions and the proof of Theorem \ref{theo1}}\label{2}
\subsection{Liapunov-Schmidt Reduction}\label{2.1}
\setcounter{equation}{0}
In this section, we use the standard Liapunov-Schmidt reduction procedure to solve problem (\ref{p-1}). Since this has become a  rather routine procedure, we omit most of the proofs. (The only part we need to pay attention to is the independence of all the coefficients on the number of spikes $k$.) We refer \cite{AWZ}, \cite{MPW} and \cite{LNW}  for technical details.

Let $ \eta \in (0,1)$ and we define
\begin{equation}
\label{Wdef}
W : =  \sum_{\mathbf{Q} \in \Lambda_k} e^{- \eta |\cdot - Q_i |}.
\end{equation}
 Consider the norm
\begin{equation}\label{e302}
 \quad \| h \|_{*} =\sup_{x \in \R^N}|   W(x)^{-1}  h(x) |
\end{equation}
where $(Q_1,\cdots,Q_k) \in \Lambda_k$ and $\Lambda_k$ is defined in (\ref{lambda}).

We first estimate the error in the above norm.

\begin{lemma}\label{lemma201}
Given $0<\eta<1$. For $\rho$ large enough, and any $\mathbf{Q}_k\in\Lambda_k$, $\delta <e^{-2\rho}$, the following estimate holds:
\begin{equation}
\|S(w_{\mathbf{Q}_k})\|_*\leq ce^{-\xi \rho},
\end{equation}
for some constant $\xi>0$ and $c$ independent of $\rho$, $k$ and $\mathbf{Q}_k$.
\end{lemma}
\begin{proof}
Observe that
\begin{equation}
S(w_{\mathbf{Q}_k})=-\delta Vw_{\mathbf{Q}_k}+f(w_{\mathbf{Q}_k})-\sum_{i=1}^kf(w_{Q_i}).
\end{equation}

Firstly, fix $j \in \{ 1 , \ldots , k\}$ and consider the region $|x-Q_j | \leq {\rho \over
2} $. In this region we have
\begin{eqnarray}
|f(w_{\mathbf{Q}_k})-\sum_{i=1}^kf(w_{Q_i})|&\leq & C \left[ f'(w_{Q_j}) \sum_{Q_i \not= Q_j}
w(x-Q_i ) + \sum_{Q_i \not= Q_j }f(w_{Q_i}) \right] \nonumber\\
&\leq & C( f'(w_{Q_j})e^{-{1\over 2} \rho}+e^{-\frac{(1+\sigma)\rho}{2}}) \label{EE1} \\
\nonumber \\
&\leq & C e^{-\xi \rho}e^{-\eta |x-Q_j|} \nonumber
\end{eqnarray}
for a proper choice of $\xi >0$.

Consider now the region $|x-Q_j | > {\rho \over 2}$, for
all $j $. We get in the region under consideration
\begin{eqnarray}
|f(w_{\mathbf{Q}_k})-\sum_{i=1}^kf(w_{Q_i}) | &\leq & C \left[ \sum_{j } f(w_{Q_j} ) \right]
\leq C \left[ \sum_{j } e^{-(1+\sigma) |x-Q_j |}  \right]\nonumber
 \\
&\leq &  \sum_{j} e^{-\eta |x-Q_j |}  e^{-{1+\sigma-\eta
\over 2} \rho  } \label{EE2}\\
&\leq & ce^{-\xi \rho}\sum_je^{-\eta |x-Q_j |}\nonumber
\end{eqnarray}
for some $\xi >0$.

Secondly, it is easy to see that under the assumption on $\delta$
\begin{equation}\label{EE3}
|\delta Vw_{\mathbf{Q}_k}|\leq c e^{-\xi \rho}\sum_je^{-\eta |x-Q_j |} ,
\end{equation}
for some $\xi>0$.

From the above estimates (\ref{EE1}), (\ref{EE2}) and (\ref{EE3}), we get that
\begin{equation}
\|S(w_{\mathbf{Q}_k})\|_*\leq ce^{-\xi \rho}
\end{equation}
for some $\xi>0$ independent of $\rho, k $ and $\mathbf{Q}_k$.
\end{proof}

The following proposition is standard. We refer to \cite{LNW} and further improvements of \cite{AWZ}.

\begin{proposition} \label{p401}
Given $0<\eta<1$. There exist positive numbers $\rho_0$, $C$ and $\xi >0$ such that for all $\rho\geq \rho_0$, and for any $\mathbf{Q}_k\in\Lambda_k$, $\delta <e^{-2\rho}$, there is a unique solution $(\phi_{\mathbf{Q}_k} , \{c_{ij}\}  )$  to the following problem:
 \begin{equation}
 \left\{\begin{array}{c}
 \Delta (w_{\mathbf{Q}_k}+\phi_{\mathbf{Q}_k})-(1+\delta V)(w_{\mathbf{Q}_k}+\phi_{\mathbf{Q}_k})+f(w_{\mathbf{Q}_k}+\phi_{\mathbf{Q}_k})=\sum_{i=1,\cdots,k, j=1,\cdots, N}c_{ij}Z_{ij},\\
 \int_{\R^N}\phi_{\mathbf{Q}_k}Z_{ij}dx=0 \mbox{ for } i=1,\cdots,k, j=1,\cdots,N.
 \end{array}
 \right.
 \end{equation}

 Furthermore $\phi_{\mathbf{Q}_k}$ is $C^1$ in $\mathbf{Q}_k$ and we have
\begin{equation}
 \|\phi_{\mathbf{Q}_k} \|_{*} \leq C\|S(w_{\mathbf{Q}_k})\|_*\leq C e^{-\xi \rho }, |c_{ij}|\leq ce^{-\xi\rho}.
\label{est2}\end{equation}

\end{proposition}

\subsection{A secondary Liapunov-Schmidt reduction}\label{2.2}
\setcounter{equation}{0}

In this section, we present a key estimate on the difference between the solutions in the $k-$th step and $ (k+1)-$th step. This second Liapunov-Schmidt reduction has been used in the paper \cite{AWZ}.

For $(Q_1,\cdots,Q_{k})\in \Lambda_{k}$,  we  denote $u_{ Q_1, \cdots, Q_k}$ as $ w_{ Q_1,..., Q_k}+ \phi_{Q_1, ..., Q_k}$, where $ \phi_{ Q_1, \cdots, Q_k}$ is the unique solution given by Proposition \ref{p401}. The main estimate below states that the difference between $ u_{Q_1, \cdots, Q_{k+1}}$ and $ u_{ Q_1, \cdots, Q_k}+ w_{ Q_{k+1}}$ is small globally in $H^1 (\R^N)$ norm.

To this end, we now write
 \begin{eqnarray}
 u_{Q_1,\cdots,Q_{k+1}}&=&u_{Q_1,\cdots,Q_k}+w_{Q_{k+1}}+\varphi_{k+1}\\
 &=&\bar{W}+\varphi_{k+1}\nonumber,
 \end{eqnarray}
where
$$\bar{W}= u_{Q_1,\cdots,Q_k}+w_{Q_{k+1}}.$$

By Proposition \ref{p401}, we can easily derive that
\begin{equation}
\label{vark100}
 \|\varphi_{k+1} \|_{*} \leq C e^{-\xi \, \rho }.
\end{equation}

However the estimate (\ref{vark100}) is not sufficient.   We need the following key estimate for $\varphi_{k+1}$. (In the following we will always assume that $\eta>\frac{1}{2}$.)

 \begin{lemma}\label{lemma501}
 Let $\rho$, $\delta$ be as in Proposition \ref{p401}. Then it holds
 \begin{eqnarray}
\label{keyvar}
 \int_{\R^N} (|\nabla \varphi_{k+1}|^2 + \varphi_{k+1}^2 ) &\leq& Ce^{-\xi  \rho }\sum_{i=1}^kw(|Q_{k+1}-Q_i|)\\
 &+& C \delta^2(\int_{\R^N}V^2w^2_{Q_{k+1}}dx+(\int_{\R^N}|V|w_{Q_{k+1}}dx)^2),\nonumber
 \end{eqnarray}
 for some constant $C>0,\xi>0$ independent of $\rho ,k, \eta$ and $\mathbf{Q}_{k+1}\in \Lambda_{k+1}$.
 \end{lemma}

Before we proceed with the proof, we need the following lemma which can be found in Lemma 2.3 in \cite{BL}.

\begin{lemma}
 For $|Q_i-Q_j|\geq \rho$ large, it holds that
 \begin{equation}
 \int_{\R^N}f(w(x-Q_i))w(x-Q_j)dx=(\gamma_1+e^{-\xi\rho})w(|Q_i-Q_j|)
 \end{equation}
for some $\xi>0$ independent of large $\rho$ and
 \begin{equation}\label{gamma1}
 \gamma_1=\int_{\R^N}f(w)e^{-y_1}dy>0.
 \end{equation}
 \end{lemma}

\medskip
\noindent {\it Proof of Lemma \ref{lemma501}.} To prove (\ref{keyvar}), we need to perform a further decomposition.

 As we mentioned before, under the assumptions $(f_1)-(f_2)$, there exists a unique positive eigenvalue with eigenfunction $\phi_0$ of the following linearized operator:
\begin{equation}
\Delta \phi-\phi+ f'(w)\phi=\lambda_1\phi
\end{equation}
which is even and has exponential decay. We fix $ \phi_0$ such that $ \max_{ y \in \R^N} \phi_0 =1$. Denote by $\phi_i=\chi_i\phi_0 (x-Q_i)$, where $ \chi_i$ is the cut-off function introduced in Section 1.2.

By the equations satisfied by $\varphi_{k+1}$, we have
\begin{equation}\label{varphi}
\bar{L}\varphi_{k+1}=\bar{S}+\sum_{i=1,\cdots,k+1, j=1,\cdots,N}c_{ij}Z_{ij}
\end{equation}
for some constants $\{c_{ij}\}$, where
\begin{equation*}
\bar{L}=\Delta-(1+\delta V)+f'(\tilde{W}),
\end{equation*}
\begin{equation*}
f'(\tilde{W})=\left\{\begin{array}{l}
\frac{ f(\bar{W} + \varphi_{k+1})- f(\bar{W})}{ \varphi_{k+1}}, \ \mbox{if} \ \varphi_{k+1} \not =0\\
f'(\bar{W}), \ \mbox{if} \ \varphi_{k+1}=0,
\end{array}
\right.
\end{equation*}
and
\begin{equation*}
\bar{S}=f(u_{Q_1,\cdots,Q_{k}}+w_{Q_{k+1}})-f(u_{Q_1,\cdots,Q_k})-f(w_{Q_{k+1}})-\delta V w_{Q_{k+1}}.
\end{equation*}

We proceed the proof in a few steps.

The $L^2$-norm of $\bar{S}$ is estimated first:

By the estimate in Proposition \ref{p401}, we have the following estimate
\begin{eqnarray*}
&&\int_{\R^N}|f(u_{Q_1,\cdots,Q_{k}}+w_{Q_{k+1}})-f(u_{Q_1,\cdots,Q_k})-f(w_{Q_{k+1}})|^2dx\\
&&\leq ce^{-\xi\rho}\sum_{i=1}^kw(|Q_{k+1}-Q_i|),
\end{eqnarray*}
and the last term can be estimated as
\begin{eqnarray*}
&&\int_{\R^N}(\delta V w_{Q_{k+1}})^2dx \leq C\delta^2\int_{\R^N}V^2w^2_{Q_{k+1}}dx.
\end{eqnarray*}
 So by the above two estimates, we have
\begin{equation}\label{s}
\|\bar{S}\|^2_{L^2(\R^N)}\leq C(e^{-\xi\rho}\sum_{i=1}^kw( |Q_{k+1}-Q_i|)+\delta^2\int_{\R^N}V^2w^2_{Q_{k+1}}dx).
\end{equation}

By the estimate (\ref{vark100}), we have  the following estimate
\begin{equation}
\label{W123}
 \tilde{W}= \sum_{i=1}^{k+1} w(x-Q_i) + O(e^{- (1+\xi)\frac{\rho}{2}}).
\end{equation}

Decompose $\varphi_{k+1}$ as
\begin{equation}\label{decom}
\varphi_{k+1}=\psi+\sum_{i=1}^{k+1}c_i\phi_i
+\sum_{i=1,\cdots,k+1,j=1,\cdots,N}d_{ij}Z_{ij}
\end{equation}
for some $c_i,d_{ij}$ such that
\begin{equation}
\label{345n}
\int_{\R^N} \psi\phi_idx=\int_{\R^N}\psi Z_{ij}dx=0,\ i=1,..., k, \ j=1,..., N.
\end{equation}

Since
\begin{equation}
\varphi_{k+1}=\phi_{Q_1,\cdots,Q_{k+1}}-\phi_{Q_1,\cdots,Q_k},
\end{equation}

we have for $i=1,\cdots,k$,
\begin{eqnarray*}
d_{ij}&=&\int_{\R^N} \varphi_{k+1}Z_{ij}\\
&=&\int_{\R^N}(\phi_{Q_1,\cdots,Q_{k+1}}
-\phi_{Q_1,\cdots,Q_k})Z_{ij}\\
&=&0
\end{eqnarray*}
and
\begin{eqnarray*}
d_{k+1,j}&=&\int_{\R^N} \varphi_{k+1}Z_{k+1,j}\\
&=&\int_{\R^N}(\phi_{Q_1,\cdots,Q_{k+1}}-\phi_{Q_1,\cdots,Q_k})Z_{k+1,j}\\
&=&-\int_{\R^N} \phi_{Q_1,\cdots,Q_k}Z_{k+1,j},
\end{eqnarray*}
where we use the orthogonality conditions satisfied by $\phi_{Q_1,\cdots,Q_k}$ and $\phi_{Q_1,\cdots,Q_{k+1}}$.
So by Proposition \ref{p401}, we have
\begin{equation}\label{d}
\left\{\begin{array}{l}
|d_{ij}|=0 \mbox{ for }i=1,\cdots,k,\\
\\
|d_{k+1,j}|\leq ce^{-\xi\rho }\sum_{i=1}^ke^{-\eta|Q_i-Q_{k+1}|}.
\end{array}
\right.
\end{equation}

By (\ref{decom}), we can rewrite (\ref{varphi}) as
\begin{equation}\label{decom1}
\bar{L}\psi+\sum_{i=1}^{k+1} c_i\bar{L}(\phi_i)+\sum_{i=1,\cdots,k+1,j=1,\cdots,N}d_{ij}\bar{L}Z_{ij}
=\bar{S}+\sum_{i=1,\cdots,k+1,j=1,\cdots,N}c_{ij}Z_{ij}.
\end{equation}

To obtain the estimates for the coefficients $c_i$ , we use the equation (\ref{decom1}).

First, multiplying (\ref{decom1}) by $\phi_i$ and integrating over $\R^N$, we have
\begin{eqnarray}
\label{W345}
c_i\int_{\R^N}(\bar{L}(\phi_i))\phi_i&=&-\sum_{j=1}^N d_{ij}\int_{\R^N} \bar{L}(Z_{ij})\phi_i\\
&+&\int_{\R^N}\bar{S}\phi_i-\int_{\R^N}(\bar{L}\psi)\phi_i\nonumber
\end{eqnarray}
where
\begin{equation}
\left\{\begin{array}{ll}
\label{W346}
|\int_{\R^N}\bar{S}\phi_i|\leq ce^{-\xi \rho} e^{-\eta|Q_i-Q_{k+1}|}+\delta |\int_{\R^N}Vw_{Q_{k+1}}\phi_i dx| \mbox{ for }i=1,\cdots,k\\
\\
|\int_{\R^N}\bar{S}\phi_{k+1}|\leq ce^{-\xi\rho}\sum_{i=1}^k e^{-\eta|Q_i-Q_{k+1}|}+\delta |\int_{\R^N}Vw_{Q_{k+1}}\phi_{k+1} dx|.
\end{array}
\right.
\end{equation}

 From (\ref{W123}) we see that
\begin{equation}
\label{W234}
\int_{\R^N}(\bar{L}\phi_i)\phi_i =  - \lambda_1 \int_{\R^n} \phi_0^2 + O(e^{-(1+\xi)\frac{\rho}{2}}).
\end{equation}

Combining (\ref{d}) and (\ref{W345})-(\ref{W234}), and the orthogonal conditions satisfied by $\psi$, we have
\begin{equation}\label{c}
\left\{\begin{array}{ll}
|c_i|\leq  ce^{-\xi \rho} e^{-\eta|Q_i-Q_{k+1}|}+\delta |\int_{\R^N}Vw_{Q_{k+1}}\phi_i dx|+e^{-\xi \rho}\|\psi\|_{H^1(B_{\frac{\rho}{2}}(Q_i))}, \ i=1,..., k\\
\\
|c_{k+1}|\leq  ce^{-\xi \rho}\sum_{i=1}^k e^{-\eta|Q_i-Q_{k+1}|}+\delta |\int_{\R^N}Vw_{Q_{k+1}}\phi_{k+1} dx|+e^{-\xi \rho}\|\psi\|_{H^1(B_{\frac{\rho}{2}}(Q_{k+1}))}.
\end{array}
\right.
\end{equation}

Next let us estimate $\psi$. Multiplying (\ref{decom1}) by $\psi$ and integrating over $\R^N$, we find
\begin{eqnarray}\label{psi}
\int_{\R^N} \bar{L}(\psi)\psi&=&\int_{\R^N}\bar{S} \psi-\sum_{i=1,\cdots,k+1,j=1,\cdots,N}d_{ij}\int_{\R^N} \bar{L}(Z_{ij})\psi\\
&-&\sum_{i=1}^{k+1}c_i\int_{\R^N}(\bar{L}\phi_i)\psi.\nonumber
\end{eqnarray}
We claim that
\begin{equation}
\int_{\R^N} [-\bar{L}(\psi)\psi] \geq c_0 \|\psi\|^2_{H^1(\R^N)}
\end{equation}
for some constant $c_0>0$.

Since the approximate solution is exponentially decay away from the points $Q_i$, we have
\begin{equation}
\int_{\R^N \backslash  \cup_i B_{\frac{\rho-1}{2}}(Q_i)}\bar{L}(\psi)\psi\geq \frac{1}{2}
\int_{\R^N \backslash \cup_i B_{\frac{\rho-1}{2}}(Q_i)}|\nabla \psi|^2+|\psi|^2.
\end{equation}
Now we only need to prove the above estimates in the domain $\cup_i B_{\frac{\rho-1}{2}}(Q_i)$. We prove it by contradiction. Otherwise, there exists a sequence $\rho_n\to +\infty$, and $Q_i^{(n)}$ such that
\begin{eqnarray*}
\int_{B_{\frac{\rho_n-1}{2}}(Q_i^{(n)})}|\nabla \psi_n|^2+|\psi_n|^2=1,\ \int_{B_{\frac{\rho_n-1}{2}}(Q_i^{(n)})}\bar{L}(\psi_n)\psi_n\to 0,\mbox{ as } n\to \infty.
\end{eqnarray*}
Then we can extract from the sequence $\psi_n(\cdot-Q_i^{(n)})$ a subsequence which will converge weakly in $H^1(\R^N)$ to $\psi_\infty$, such that
\begin{equation}\label{phi1}
\int_{\R^N}|\nabla \psi_\infty|^2+|\psi_\infty|^2-f'(w)\psi_\infty^2=0,
\end{equation}
and
\begin{equation}\label{phi2}
\int_{\R^N} \psi_\infty \phi_0=\int_{\R^n} \psi_\infty \frac{\partial w}{\partial x_i}=0, \mbox{ for }i=1,\cdots,N.
\end{equation}
From (\ref{phi1}) and (\ref{phi2}), we deduce that $\psi_\infty=0$.

Hence
\begin{equation}
\psi_n\rightharpoonup 0 \mbox{ weakly } \mbox{ in } H^1(\R^N).
\end{equation}
So
\begin{equation}
\int_{B_{\frac{\rho_n-1}{2}}(Q_i^{(n)})} f'(\tilde{W})\psi_n^2\to 0 \mbox{ as }n\to \infty.
\end{equation}
We have
\begin{equation}
\|\psi_n\|_{H^1(B_{\frac{\rho_n-1}{2}})}\to 0 \mbox{ as }n\to \infty.
\end{equation}
This contradicts the assumption
\begin{equation}
\|\psi_n\|_{H^1}=1.
\end{equation}

So we get that
\begin{equation}\label{psi1}
\int_{\R^N} [- \bar{L}(\psi)\psi ] \geq c_0 \|\psi\|^2_{H^1(\R^N)}.
\end{equation}

From (\ref{psi}) and (\ref{psi1}), we get
\begin{eqnarray}
\|\psi\|^2_{H^1 (\R^N) }&\leq& c(\sum_{ij}|d_{ij} | |\int_{\R^N} \bar{L}(Z_{ij}) \psi | +\sum_{i}|c_i||\int_{\R^N}(\bar{L}\phi_i)\psi|+|\int_{\R^N} \bar{S} \psi | ) \nonumber\\
&\leq &c(\sum_{ij}|d_{ij}|\|\psi\|_{H^1 (\R^N) }+\sum_i |c_i|\|\psi\|_{H^1 (B_{\frac{\rho}{2}}(Q_i)) }\nonumber\\
&+&\|\bar{S}\|_{L^2 (\R^N) }\|\psi\|_{H^1 (\R^N)}).
\end{eqnarray}
So by estimate (\ref{c}) and the above,
\begin{eqnarray}\label{psi2}
\|\psi\|_{H^1 (\R^N) }
&\leq& c(\sum_{ij}|d_{ij}|+e^{-\xi\rho}\sum_{i=1}^ke^{-\eta |Q_{k+1}-Q_i|}\\
&+&\delta \int_{\R^N}|V|w_{Q_{k+1}}dx +\|\bar{S}\|_{L^2 (\R^N)}).\nonumber
\end{eqnarray}
From (\ref{d}) (\ref{s}) and (\ref{psi2}), and recall that $\eta>\frac{1}{2}$,  we get that
\begin{eqnarray}\label{varphi1}
\|\varphi_{k+1}\|_{H^1 (\R^N) }&\leq& C(e^{-\xi\rho}\sum_{i=1}^ke^{-\eta |Q_{k+1}-Q_i|}+\delta \int_{\R^N}|V|w_{Q_{k+1}}dx+\|\bar{S}\|_{L^2})\nonumber\\
&\leq& C(e^{-\xi\rho}\sum_{i=1}^ke^{-\eta|Q_{k+1}-Q_i|}+e^{-\xi  \rho }(\sum_{i=1}^kw(|Q_{k+1}-Q_i|))^{\frac{1}{2}}\nonumber\\
&+&\delta\int_{\R^N}|V|w_{Q_{k+1}}dx+\delta(\int_{\R^N}V^2w^2_{Q_{k+1}}dx)^{\frac{1}{2}}).
\end{eqnarray}

Since we choose $\eta>\frac{1}{2}$, by the definition of the configuration space,  we have
\begin{equation}\label{ineq}
(\sum_{i=1}^ke^{-\eta|Q_i-Q_{k+1}|})^2\leq c\sum_{i=1}^kw(|Q_i-Q_{k+1}|).
\end{equation}

By (\ref{varphi1}) and (\ref{ineq}), we thus obtain that
\begin{eqnarray}
\|\varphi_{k+1}\|_{H^1 (\R^N) }
&\leq& C(e^{-\xi  \rho }(\sum_{i=1}^kw(|Q_{k+1}-Q_i|))^{\frac{1}{2}}
+\delta\int_{\R^N}|V|w_{Q_{k+1}}dx\nonumber\\
&+&\delta(\int_{\R^N}V^2w^2_{Q_{k+1}}dx)^{\frac{1}{2}}).
\end{eqnarray}
The estimate (\ref{keyvar}) then follows.

Moreover, from the estimate (\ref{c}) and (\ref{d}), and taking into consideration that $\chi_i$ is supported in $B_{\frac{\rho}{2}}(Q_i)$, using holder inequality, we can get a more accurate estimate on $\varphi_{k+1}$,

\begin{eqnarray}\label{varphi2}
\|\varphi_{k+1}\|_{H^1 (\R^N) }
&\leq& C(e^{-\xi  \rho }(\sum_{i=1}^kw(|Q_{k+1}-Q_i|))^{\frac{1}{2}}
+\delta\sum_{i=1,\cdots,k+1}(\int_{B_{\frac{\rho}{2}}(Q_i)}V^2w^2_{Q_{k+1}}dx)^{\frac{1}{2}}\nonumber\\
&+&\delta(\int_{\R^N}V^2w^2_{Q_{k+1}}dx)^{\frac{1}{2}}).
\end{eqnarray}

\qed

\subsection{The Reduced Problem: A Maximization Procedure}\label{2.3}
\setcounter{equation}{0}
In this section, we study a maximization problem. Fix $\mathbf{Q}_k\in \Lambda_k$, we define a new functional
\begin{equation}
\mathcal{M}(\mathbf{Q}_k)=J(u_{\mathbf{Q}_k})=J[w_{\mathbf{Q}_k}+\phi_{\mathbf{Q}_k}]: \Lambda_k \rightarrow \R.
\end{equation}
Define
\begin{equation}
\mathcal{C}_k=\mbox{sup}_{\mathbf{Q}_k\in\Lambda_k}\{\mathcal{M}(\mathbf{Q}_k)\}.
\end{equation}

Note that $\mathcal{M}(\mathbf{Q}_k)$ is continuous in $\mathbf{Q}_k$. We will show below that  the maximization problem has a solution. Let $\mathcal{M}(\bar{\mathbf{Q}}_k)$ be the maximum where $\bar{\mathbf{Q}}_k=(\bar{Q}_1,\cdots,\bar{Q}_k)\in \bar{\Lambda}_k$, that is
 \begin{equation}
 \mathcal{M}(\bar{Q}_1,\cdots,\bar{Q}_k)=\max_{\mathbf{Q}_k\in \Lambda_k}\mathcal{M}(\mathbf{Q}_k),
 \end{equation}
 and we denote the solution by $u_{\bar{Q}_1,\cdots,\bar{Q}_k}$.

 We first prove that the maximum can be attained at finite points for each $\mathcal{C}_k$.

\begin{lemma}\label{lemma601}
Let assumptions $(H1)-(H2)$ and the assumptions in Proposition \ref{p401} be satisfied. Then, for all $k$:
\begin{itemize}
\item
There exists $\mathbf{Q}_k=(Q_1,Q_2,\cdots,Q_k)\in \Lambda_k$ such that
\begin{equation}
\mathcal{C}_k=\mathcal{M}(\mathbf{Q}_k);
\end{equation}
\item
There holds
\begin{equation}\label{e501}
\mathcal{C}_{k+1}>\mathcal{C}_k+I(w),
\end{equation}
where $I(w)$ is the energy of $w$,
\begin{equation}
I(w)=\frac{1}{2}\int_{\R^N}(|\nabla w|^2+w^2)-\int_{\R^N}F(w)dx.
\end{equation}
\end{itemize}
\end{lemma}

\begin{proof} In this part, we follow the proofs  in \cite{CPS} but we use the estimates we derived in Section 3. We divide the proof into several steps.

\medskip

\noindent
{\bf Step 1:}
$\mathcal{C}_1>I(w)$, and $\mathcal{C}_1$ can be attained at finite point. First using standard Liapunov-Schmidt reduction, we have
\begin{equation}
\|\phi_Q\|_{H^1}\leq c\|\delta V w_Q\|_{L^2}.
\end{equation}

Assuming that $|Q|\to \infty$, then we have
\begin{eqnarray*}
&&J(u_Q)=\frac{1}{2}\int_{\R^N}|\nabla u_Q|^2+u_Q^2-\int_{\R^N}F(u_Q)dx\\
&&+\frac{1}{2}\int_{\R^N}\delta Vu_Q^2dx\\
&&\geq I(w)+\frac{1}{2}\int_{\R^N}\delta V w_Q^2dx+\|\phi_Q\|^2_{H^1}\\
&&\geq I(w)+\frac{1}{2}\int_{\R^N}\delta V w_Q^2dx+\int_{\R^N}\delta^2 V^2 w_Q^2dx\\
&&\geq I(w)+\frac{1}{4}[\int_{B_{\frac{\rho}{2}}(Q)}\delta V w_Q^2 dx -\sup_{B_{\frac{|Q|}{4}}(0)}|w_Q|^2\int_{supp V^-}\delta|V|dx]\\
&&\geq I(w)+\frac{1}{4}\int_{B_{\frac{\rho}{2}}(Q)}\delta V w_Q^2 dx-O(e^{-\frac{3}{2} |Q|}).
\end{eqnarray*}

By the slow decay assumption on the potential $V$, we get that
\begin{equation*}
\frac{1}{4}\int_{B_{\frac{\rho}{2}}(Q)}\delta V w_Q^2 dx-O(e^{-\frac{3}{2} |Q|})>0, \ \mbox{for} \ \ |Q| \ \mbox{large},
\end{equation*}
so
\begin{equation*}
\mathcal{C}_1\geq J(u_Q)>I(w).
\end{equation*}

Let us prove now that $\mathcal{C}_1$ can be attained at finite point. Let $\{Q_i\}$ be a sequence such that
$\lim_{i\to \infty}\mathcal{M}(Q_i)=\mathcal{C}_1$, and assume that $|Q_i|\to \infty$,
\begin{eqnarray*}
&&J(u_{Q_i})\\
&&=\frac{1}{2}\int_{\R^N}|w_{Q_i}+\phi_{Q_i}|^2+|w_{Q_i}+\phi_{Q_i}|^2dx-\int_{\R^N}F(w_{Q_i}+\phi_{Q_i})dx
\\
&&+\frac{1}{2}
\int_{\R^N}\delta V(w_{Q_i}+\phi_{Q_i})^2 dx\\
&&=\frac{1}{2}\int_{\R^N}|\nabla w_{Q_i}|^2+|w_{Q_i}|^2dx-\int_{\R^N}F(w_{Q_i})dx\\
&&+\frac{1}{2}\int_{\R^N}|\nabla \phi_{Q_i}|^2+|\phi_{Q_i}|^2dx+\int_{\R^N}\nabla w_{Q_i}\nabla \phi_{Q_i}+w_{Q_i}\phi_{Q_i}-f(w_{Q_i})\phi_{Q_i} dx\\
&&-\int_{\R^N}F(w_{Q_i}+\phi_{Q_i})-F(w_{Q_i})-f(w_{Q_i})\phi_{Q_i}dx+\frac{1}{2}\int_{\R^N}\delta V(w_{Q_i}+\phi_{Q_i})^2dx\\
&&\leq I(w)+c\|S(w_{Q_i})\|^2_{L^2(\R^N)}+\frac{1}{2}\int_{\R^N}\delta V(w_{Q_i}+\phi_{Q_i})^2dx\\
&&\leq I(w)+O(\int_{\R^N}\delta^2 V^2w_{Q_i}^2dx)+\frac{1}{2}\int_{\R^N}\delta V(w_{Q_i}+\phi_{Q_i})^2dx.
\end{eqnarray*}
Since $V(x)\to 0$ as $|x|\to \infty$, we have
\begin{eqnarray*}
O(\int_{\R^N}\delta^2 V^2w_{Q_i}^2dx)+\frac{1}{2}\int_{\R^N}\delta V(w_{Q_i}+\phi_{Q_i})^2dx\to 0\mbox{ as }i\to \infty.
\end{eqnarray*}
So we have
\begin{eqnarray*}
\mathcal{C}_1=\lim_{i\to \infty }J(u_{Q_i})\leq I(w).
\end{eqnarray*}
A contradiction. Thus $\mathcal{C}_1$ can be attained at a finite point.

\medskip

\noindent
{\bf Step 2:}
Assume that there exists $\mathbf{Q}_k=(\bar{Q}_1,\cdots,\bar{Q}_k)\in \Lambda_k$ such that $\mathcal{C}_k=\mathcal{M}(\mathbf{Q}_k)$, and we denote the solution by $u_{\bar{Q}_1,\cdots, \bar{Q}_k}$,

Next, we prove that there exists $(Q_1,\cdots,Q_{k+1})\in \Lambda_{k+1}$ such that $\mathcal{C}_{k+1}$ can be attained.

Let $((Q_1^{(n)},\cdots,Q_{k+1}^{(n)}))_n$ be a sequence such that
\begin{equation}\label{ck1}
\mathcal{C}_{k+1}=\lim_{n\to \infty }\mathcal{M}(Q_1^{(n)},\cdots,Q_{k+1}^{(n)}).
\end{equation}

We claim that $(Q_1^{(n)},\cdots,Q_{k+1}^{(n)})$ is bounded. We prove it by contradiction. Without loss of generality, we assume that $|Q_{k+1}^{(n)}|\to \infty$ as $n\to \infty$. In the following we omit the index $n$ for simplicity.

\begin{eqnarray}\label{eq601}
&&J(u_{Q_1,\cdots,Q_{k+1}})\\
&&=J(u_{Q_1,\cdots,Q_k}+w_{Q_{k+1}}+\varphi_{k+1})\nonumber\\
&&=J(u_{Q_1,\cdots,Q_k}+w_{Q_{k+1}})\nonumber\\
&&+\frac{1}{2}\int_{\R^N}|\nabla \varphi_{k+1}|^2+|\varphi_{k+1}|^2+\delta V\varphi_{k+1}^2dx\nonumber\\
&&+\int_{\R^N}\nabla(u_{Q_1,\cdots,Q_k}+w_{Q_{k+1}})\nabla \varphi_{k+1}+(1+\delta V)(u_{Q_1,\cdots,Q_k}+w_{Q_{k+1}})\varphi_{k+1}\nonumber\\
&&-f(u_{Q_1,\cdots,Q_k}+w_{Q_{k+1}})\varphi_{k+1} dx\nonumber\\
&&-\int_{\R^N}F(u_{Q_1,\cdots,Q_k}+w_{Q_{k+1}}+\varphi_{k+1})-F(u_{Q_1,\cdots,Q_k}+w_{Q_{k+1}})\nonumber\\
&&-f(u_{Q_1,\cdots,Q_k}+w_{Q_{k+1}})\varphi_{k+1}
dx\nonumber\\
&&=J(u_{Q_1,\cdots,Q_k}+w_{Q_{k+1}})+O(\|\varphi_{k+1}\|^2_{H^1(\R^N)}\nonumber\\
&&+\|\bar{S}(u_{Q_1,\cdots,Q_k}+w_{Q_{k+1}})\|_{H^1(\R^N)}
\|\varphi_{k+1}\|_{H^1(\R^N)})+\sum_{i=1,\cdots,k,j=1,\cdots,N}\int_{\R^N}c_{ij}Z_{ij}\varphi_{k+1}dx\nonumber\\
&&=J(u_{Q_1,\cdots,Q_k}+w_{Q_{k+1}})\nonumber\\
&&+O(e^{-\xi\rho}\sum_{i=1}^kw(|Q_{k+1}- Q_i|)+\delta^2(\int_{\R^N}V^2w_{Q_{k+1}}^2dx+(\int_{\R^N}Vw_{Q_{k+1}}dx)^2),\nonumber
\end{eqnarray}

and
\begin{eqnarray}\label{eq602}
&&J(u_{Q_1,\cdots,Q_k}+w_{Q_{k+1}})\\
&&=J(u_{Q_1,\cdots,Q_k})+I(w_{Q_{k+1}})+\frac{1}{2}\int_{\R^N}\delta V w^2_{Q_{k+1}}dx\nonumber\\
&&+\int_{\R^N}\nabla u_{Q_1,\cdots,Q_k}\nabla w_{Q_{k+1}}+(1+\delta V)u_{Q_1,\cdots,Q_k}w_{Q_{k+1}}dx\nonumber\\
&&-\int_{\R^N}F(u_{Q_1,\cdots,Q_k}+w_{Q_{k+1}})-F(u_{Q_1,\cdots,Q_k})-F(w_{Q_{k+1}})dx\nonumber\\
&&\leq \mathcal{C}_k+I(w)+\frac{1}{2}\int_{\R^N}\delta V w^2_{Q_{k+1}}dx\nonumber\\
&&+\int_{\R^N}(f(u_{Q_1,\cdots,Q_k})-\sum_{i=1}^kc_{ij}Z_{ij})w_{Q_{k+1}}dx\nonumber\\
&&-\int_{\R^N}f(u_{Q_1,\cdots,Q_k})w_{Q_{k+1}}+f(w_{Q_{k+1}})u_{Q_1,\cdots,Q_k}dx+O(e^{-\xi\rho}
\sum_{i=1}^kw(|Q_{k+1}-Q_i|))\nonumber\\
&&\leq  \mathcal{C}_k+I(w)+\frac{1}{2}\int_{\R^N}\delta V w^2_{Q_{k+1}}dx\nonumber\\
&&-\int_{\R^N}\sum_{i=1}^k c_{ij}Z_{ij}w_{Q_{k+1}}dx-\int_{\R^N}f(w_{Q_{k+1}})(w_{Q_1,\cdots,Q_k}
+\phi_{\mathbf{Q}_k})\nonumber\\
&&+O(e^{-\xi\rho}
\sum_{i=1}^kw(|Q_{k+1}-Q_i|)).\nonumber
\end{eqnarray}

By estimate (\ref{est2}), and that the definition of $Z_{ij}$,  we have
\begin{eqnarray}\label{eq603}
|\sum_{i=1}^kc_{ij}\int_{\R^N}Z_{ij}w_{Q_{k+1}}dx|\leq ce^{-\xi \rho}\sum_{i=1}^kw(|Q_{k+1}-Q_i|).
\end{eqnarray}

By the equation satisfied by $\phi_k$
\begin{eqnarray*}
\Delta \phi_{\mathbf{Q}_k}-\phi_{\mathbf{Q}_k}+f'(w_{\mathbf{Q}_k})\phi_{\mathbf{Q}_k}
=S(w_{\mathbf{Q}_k})+N(\phi_{\mathbf{Q}_k})+\delta V\phi_{\mathbf{Q}_k}+\sum_{i=1,\cdots,k,j=1,\cdots,N} c_{ij}Z_{ij},
\end{eqnarray*}
where
\begin{equation}
N(\phi_{\mathbf{Q}_k})=f(w_{\mathbf{Q}_k}+\phi_{\mathbf{Q}_k})
-f(w_{\mathbf{Q}_k})-f'(w_{\mathbf{Q}_k})\phi_{\mathbf{Q}_k},
\end{equation}
we derive that
\begin{eqnarray*}
&&\int_{\R^N}f(w_{Q_{k+1}})\phi_{\mathbf{Q}_k} dx\\
&&=\int_{\R^N}(\Delta -1)w_{Q_{k+1}}\phi_{\mathbf{Q}_k} dx\\
&&=\int_{\R^N}(\Delta -1)\phi_{\mathbf{Q}_k} w_{Q_{k+1}}dx\\
&&=\int_{\R^N}(S(w_{\mathbf{Q}_k})+N(\phi_{\mathbf{Q}_k})+\delta V\phi_{\mathbf{Q}_k}\\
&&+\sum_{i=1,\cdots,k,j=1,\cdots,N} c_{ij}Z_{ij}-f'(w_{\mathbf{Q}_k})\phi_{\mathbf{Q}_k} )w_{Q_{k+1}}dx.\\
\end{eqnarray*}
 We can further choose $\eta $ such that $\eta+\sigma>1$, $(1+\sigma)\eta>1$ , we can easily get that
\begin{eqnarray*}
\int_{\R^N}(N(\phi_{\mathbf{Q}_k})-f'(w_{\mathbf{Q}_k})\phi_{\mathbf{Q}_k} )w_{Q_{k+1}}dx\leq Ce^{-\xi \rho }\sum_{i=1}^kw( |Q_{k+1}-Q_i|)),
\end{eqnarray*}
and
\begin{eqnarray*}
\int_{\R^N}\sum c_{ij}Z_{ij}w_{Q_{k+1}}dx\leq ce^{-\xi \rho }\sum_{i=1}^kw(|Q_{k+1}-Q_i|),
\end{eqnarray*}

\begin{eqnarray*}
&&\int_{\R^N}(S(w_{\mathbf{Q}_k})+\delta V\phi_{\mathbf{Q}_k})w_{Q_{k+1}}dx\\
&&\leq c(\delta \int_{\R^N}Vw_{\mathbf{Q}_k}w_{Q_{k+1}}dx+\delta e^{-\xi \rho} \int_{\R^N}\sum_{i=1}^ke^{-\eta|x-Q_i|}Vw_{Q_{k+1}}dx\\
&&+e^{-\xi \rho }\sum_{i=1}^kw( |Q_{k+1}-Q_i|)).
\end{eqnarray*}

By the above four estimates, we have
\begin{eqnarray}\label{eq6031}
\int_{\R^N}f(w_{Q_{k+1}})\phi_{\mathbf{Q}_k} dx&\leq& c(\delta e^{-\xi \rho}\int_{\R^N}\sum_{i=1}^ke^{-\eta|x-Q_i|}Vw_{Q_{k+1}}dx+\delta \int_{\R^N}Vw_{\mathbf{Q}_k}w_{Q_{k+1}}dx\nonumber\\
&+&ce^{-\xi \rho }\sum_{i=1}^kw( |Q_{k+1}-Q_i|)).
\end{eqnarray}

So we have
\begin{eqnarray}\label{eq604}
&&\int_{\R^N}f(w_{Q_{k+1}})(w_{Q_1,\cdots,Q_k}+\phi_{\mathbf{Q}_k})dx\\
&&=\int_{\R^N}f(w_{Q_{k+1}})w_{Q_1,\cdots,Q_k}+O(e^{-\xi \rho}\sum_{i=1}^kw(|Q_{k+1}-Q_i|))\nonumber\\
&&+\delta e^{-\xi \rho}\int_{\R^N}\sum_{i=1}^ke^{-\eta|x-Q_i|}Vw_{Q_{k+1}}dx+\delta \int_{\R^N}Vw_{\mathbf{Q}_k}w_{Q_{k+1}}dx\nonumber\\
&&\geq \frac{1}{4}\gamma_1 \sum_{i=1}^kw(|Q_{k+1}-Q_i|)+O(e^{-\xi \rho}\sum_{i=1}^kw( |Q_{k+1}-Q_i|))\nonumber\\
&&+\delta e^{-\xi \rho}\int_{\R^N}\sum_{i=1}^ke^{-\eta|x-Q_i|}Vw_{Q_{k+1}}dx+\delta \int_{\R^N}Vw_{\mathbf{Q}_k}w_{Q_{k+1}}dx.\nonumber
\end{eqnarray}
Thus combining (\ref{eq601}), (\ref{eq602}), (\ref{eq603}) and (\ref{eq604}), we obtain
\begin{eqnarray}\label{eq605}
&&J(u_{Q_1,\cdots,Q_{k+1}})\\
&&\leq \mathcal{C}_k+I(w)+\frac{1}{2}\int_{\R^N}\delta V w^2_{Q_{k+1}}dx-\frac{1}{4}\gamma \sum_{i=1}^kw(|Q_{k+1}-Q_i|)\nonumber\\
&&+O(e^{-\xi \rho}\sum_{i=1}^kw( |Q_{k+1}-Q_i|)+\delta e^{-\xi \rho}\int_{\R^N}\sum_{i=1}^ke^{-\eta|x-Q_i|}Vw_{Q_{k+1}}dx\nonumber\\
&&+\delta \int_{\R^N}Vw_{\mathbf{Q}_k}w_{Q_{k+1}}dx+\delta^2\int_{\R^N}V^2w_{Q_{k+1}}^2
+\delta^2(\int_{\R^N}Vw_{Q_{k+1}})^2).\nonumber
\end{eqnarray}

By the assumption that $|Q_{k+1}^{(n)}|\to \infty$,
\begin{eqnarray}\label{e606}
&&\int_{\R^N}\delta V w^2_{Q^{(n)}_{k+1}}dx+\delta e^{-\xi \rho}\int_{\R^N}\sum_{i=1}^ke^{-\eta|x-Q_i|}Vw_{Q^{(n)}_{k+1}}dx+\delta \int_{\R^N}Vw_{\mathbf{Q}_k}w_{Q^{(n)}_{k+1}}dx\nonumber\\
&&+\delta^2\int_{\R^N}V^2w_{Q^{(n)}_{k+1}}^2+\delta^2(\int_{\R^N}Vw_{Q^{(n)}_{k+1}})^2\to 0 \mbox{ as }n \to \infty ,
\end{eqnarray}
and
 \begin{equation}\label{e607}
 -\frac{1}{4}\gamma \sum_{i=1}^kw(|Q_{k+1}-Q_i|)+O(e^{-\xi \rho}\sum_{i=1}^kw( |Q_{k+1}-Q_i|))<0.
  \end{equation}

Combining (\ref{ck1}), (\ref{eq605}), (\ref{e606})and (\ref{e606}), we have
\begin{equation}\label{leq}
\mathcal{C}_{k+1}\leq \mathcal{C}_k+I(w).
\end{equation}

On the other hand, since by the assumption, $\mathcal{C}_k$ can be attained at $(\bar{Q}_1,\cdots,\bar{Q}_k)$, so there exists other point $Q_{k+1}$ which is far away from the $k$ points which will be determined later. Next let's consider the solution concentrated at the points $(\bar{Q}_1,\cdots, \bar{Q_{k}}, Q_{k+1})$, and we denote the solution by $u_{\bar{Q}_1,\cdots,\bar{Q}_k,Q_{k+1}}$, then similar with the above argument, using the estimate (\ref{varphi2}) of $\varphi_{k+1}$ instead of (\ref{keyvar}), we have the following estimates:
\begin{eqnarray}
J(u_{\bar{Q}_1,\cdots,\bar{Q}_k, Q_{k+1}})&=&J(u_{\bar{Q}_1,\cdots,\bar{Q}_k})+I(w)+\frac{1}{2}\int_{\R^N}\delta V w^2_{Q_{k+1}}dx\\
&+&O(\sum_{i=1,\cdots,k+1}(\int_{B_{\frac{\rho}{2}}(Q_i)}\delta^2 V^2 w^2_{Q_{k+1}}dx)^{\frac{1}{2}})^2+O(\int_{\R^N}\delta^2 V^2 w^2_{Q_{k+1}}dx)\nonumber\\
&-&O(\sum_{i=1}^kw( |Q_{k+1}- \bar{Q}_i|))\nonumber\\
&+&O(\delta e^{-\xi \rho}\int_{\R^N}\sum_{i=1}^ke^{-\eta|x-\bar{Q}_i|}Vw_{Q_{k+1}}dx+\delta\int_{\R^N} V w_{\bar{\mathbf{Q}}_k}w_{Q_{k+1}}dx).\nonumber
\end{eqnarray}

By the asymptotic behavior of $V$ at infinity, i.e. $\lim_{|x|\to \infty}V(x)e^{\bar{\eta}|x|}=+\infty$ as $|x|\to \infty$, for some $\bar{\eta}<1$, we further choose $\eta>\bar{\eta}$, then we can choose $Q_{k+1}$ such that
\begin{equation}
|Q_{k+1}|\geq \frac{\max_{i=1}^k|\bar{Q}_i|+\ln \delta}{\eta-\bar{\eta}},
\end{equation}
then we can get that
\begin{eqnarray}
&&\frac{1}{2}\int_{\R^N}\delta V w^2_{Q_{k+1}}dx
+O(\sum_{i=1,\cdots,k+1}(\int_{B_{\frac{\rho}{2}}(Q_i)}\delta^2 V^2 w^2_{Q_{k+1}}dx)^{\frac{1}{2}})^2\\
&&+O(\int_{\R^N}\delta^2 V^2 w^2_{Q_{k+1}}dx)-O(\sum_{i=1}^kw( |Q_{k+1}- \bar{Q}_i|))\nonumber\\
&&+O(\delta e^{-\xi \rho}\int_{\R^N}\sum_{i=1}^ke^{-\eta|x-\bar{Q}_i|}Vw_{Q_{k+1}}dx+\delta\int_{\R^N} V w_{\bar{\mathbf{Q}}_k}w_{Q_{k+1}}dx)\nonumber\\
&&\geq C\delta e^{-\bar{\eta}|Q_{k+1}|}-O(\sum_{i=1,\cdots,k}e^{-{\eta|\bar{Q}_i-Q_{k+1}|}})>0.\nonumber
\end{eqnarray}
So
\begin{equation}\label{geq}
\mathcal{C}_{k+1}\geq J(u_{\bar{Q}_1,\cdots,\bar{Q}_k, Q_{k+1}})>\mathcal{C}_k+I(w).
\end{equation}
Combining (\ref{geq}) and (\ref{leq}), one get that
\begin{equation}
\mathcal{C}_k+I(w)<\mathcal{C}_{k+1}\leq \mathcal{C}_k+I(w).
\end{equation}
A contradiction. So we get that $\mathcal{C}_{k+1}$ can be attained at finite points in $\Lambda_{k+1}$.

Moreover, from the proof above, we can get a relation between $\mathcal{C}_{k+1}$ and $\mathcal{C}_k$:
\begin{equation}
\mathcal{C}_{k+1}\geq \mathcal{C}_k+I(w).
\end{equation}
\end{proof}

Next we have the following Proposition:
\begin{proposition}\label{p602}
The maximization problem
\begin{equation}
\max_{\mathbf{Q}\in \bar{\Lambda}_k} \mathcal{M}(\mathbf{Q})
\end{equation}
has a solution $\mathbf{Q} \in \Lambda_k^\circ$, i.e., the interior of $\Lambda_k$.
\end{proposition}
\begin{proof}
We prove it by contradiction again.
If $\mathbf{Q}=(\bar{Q}_1,\cdots,\bar{Q}_k) \in \partial \Lambda_k$, then there exists $(i,j)$ such that $|\bar{Q}_i-\bar{Q}_j|=\rho$.  Without loss of generality, we assume $(i,j)=(i,k)$. Then follow the estimates in (\ref{eq601}), (\ref{eq602}), (\ref{eq603}) and (\ref{eq604}), we have

\begin{eqnarray}
\mathcal{C}_{k+1}&&=J(u_{\bar{Q}_1,\cdots,\bar{Q}_k})\\
&&\leq \mathcal{C}_{k-1}+I(w)+\frac{1}{2}\int_{\R^N}\delta V w^2_{Q_k}dx\nonumber\\
&&-\frac{1}{4}\gamma \sum_{i=1}^{k-1}e^{-|Q_k-Q_i|}+O(e^{-\xi \rho}\sum_{i=1}^{k-1}e^{- |Q_k-Q_i|})+O(\delta)\nonumber\\
&&\leq \mathcal{C}_{k-1}+I(w)\nonumber\\
&&+O(\delta)-\frac{1}{4}\gamma \sum_{i=1}^{k-1}w(|Q_k-Q_i|)+O(e^{-\xi \rho}\sum_{i=1}^{k-1}w( |Q_k-Q_i|)).\nonumber
\end{eqnarray}

By the definition of the configuration set, we observe that given a a ball of size $\rho$,
 there are at most $c_N:=6^N$ number of non-overlapping balls of size $\rho$ surrounding this ball. Since $|\bar{Q}_i-\bar{Q}_k|=\rho$, we have

\begin{eqnarray*}
\sum_{i=1}^{k-1}w(|Q_k-Q_i|)=w(|Q_i-Q_k|)+\sum_{j\neq i}w(|Q_j-Q_k|)\\
\end{eqnarray*}
and
\begin{eqnarray*}
\sum_{j\neq i}w(|Q_j-Q_k|)&\leq& Ce^{-\rho}+C_Ne^{-\rho-\frac{\rho}{2}}+\cdots+C_N^je^{-\rho-\frac{j\rho}{2}}+\cdots\\
&\leq& Ce^{-\rho}\sum_{j=0}^\infty e^{j\ln C_N-\frac{\rho}{2}}\\
&\leq & Ce^{-\rho},
\end{eqnarray*}
if $c_N<e^{\frac{\rho}{2}}$, which is true for $\rho$ large enough.

So

\begin{eqnarray}
\mathcal{C}_{k+1}
&&\leq \mathcal{C}_{k-1}+I(w)+c\delta-\frac{1}{4}\gamma w(\rho)+O(e^{-(1+\xi) \rho})\\
&&< \mathcal{C}_{k-1}+I(w).\nonumber
\end{eqnarray}

This reaches a contradiction with Lemma \ref{lemma601}.
\end{proof}

\subsection{Proof of Theorem \ref{theo1}}\label{2.4}
\setcounter{equation}{0}
In this section, we apply the results in Section \ref{2.1}, \ref{2.2} and Section \ref{2.3} to prove Theorem \ref{theo1}.

\medskip

\noindent
{\it Proof of Theorem \ref{theo1}:} By Proposition \ref{p401} in Section 2, there exists $\rho_0$ such that for $\rho>\rho_0$, we have $C^1$ map which, to any $\mathbf{Q}^\circ \in \Lambda_k$, associates $\phi_{\mathbf{Q}^\circ}$ such that
\begin{equation}
S(w_{\mathbf{Q}^\circ}+\phi_{\mathbf{Q}^\circ})=\sum_{i=1,\cdots,k,j=1,\cdots,n}c_{ij}Z_{ij},\ \
\int_{\R^N}\phi_{\mathbf{Q}^\circ}Z_{ij}dx=0,
\end{equation}
for some constants $\{c_{ij}\}\in \R^{kN}$.

From Proposition \ref{p602} in Section \ref{2.3}, there is a $\mathbf{Q}\in \Lambda_k^\circ$ that achieves the maximum for the maximization problem in Proposition \ref{p602}. Let $u_{\mathbf{Q}^\circ}=w_{\mathbf{Q}^\circ}+\phi_{\mathbf{Q}^\circ}$. Then we have
\begin{equation}
D_{Q_{ij}}|_{Q_i=Q_i^\circ}\mathcal{M}(\mathbf{Q}^\circ)=0,\  \ i=1,\cdots,k, \ \ j=1,\cdots,N.
\end{equation}
Hence we have
\begin{eqnarray*}
&&\int_{\R^N}\nabla u_{\mathbf{Q}} \nabla \frac{\partial (w_{\mathbf{Q}}+\phi_{\mathbf{Q}})}{\partial Q_{ij}}|_{Q_i=Q_i^\circ}+(1+\delta V)u_{\mathbf{Q}}\frac{\partial (w_{\mathbf{Q}}+\phi_{\mathbf{Q}})}{\partial Q_{ij}}|_{Q_i=Q_i^\circ}\\
&&-f(u_{\mathbf{Q}})\frac{\partial (w_{\mathbf{Q}}+\phi_{\mathbf{Q}})}{\partial Q_{ij}}|_{Q_i=Q_i^\circ}=0,
\end{eqnarray*}
which gives
\begin{equation}\label{e601}
\sum_{i=1,\cdots,k, \ j=1,\cdots,N}c_{ij}\int_{\R^N} Z_{ij}\frac{\partial (w_{\mathbf{Q}}+\phi_{\mathbf{Q}})}{\partial Q_{sl}}|_{Q_s=Q_s^\circ}=0,
\end{equation}
for $s=1,\cdots,k, l=1,\cdots,N$.
We claim that (\ref{e601}) is a diagonally dominant system.  In fact, since $\int_{\R^N} \phi_{\mathbf{Q}}Z_{sl}dx=0$, we have that
\begin{equation*}
\int_{\R^N}Z_{sl}\frac{\partial \phi_{\mathbf{Q}}}{\partial Q_{ij}}|_{Q_i=Q_i^\circ}
=-\int_{\R^N}\phi_{\mathbf{Q}}\frac{\partial Z_{sl}}{\partial Q_{ij}}=0, \mbox{ if }s\neq i.
\end{equation*}
If $s=i$, we have
\begin{eqnarray*}
|\int_{\R^N}Z_{il}\frac{\partial \phi_{\mathbf{Q}}}{\partial Q_{ij}}|_{Q_i=Q_i^\circ}|
=|-\int_{\R^N}\phi_{\mathbf{Q}}\frac{\partial Z_{il}}{\partial Q_{ij}}|\\
\leq C\|\phi_{\mathbf{Q}}\|_*=O(e^{-\frac{\rho}{2}(1+\xi)}).
\end{eqnarray*}
For $s\neq i$, we have
\begin{equation*}
\int_{\R^N}Z_{sl}\frac{\partial w_{\mathbf{Q}}}{\partial Q_{ij}}=O(e^{-\frac{|Q_i-Q_s|}{2}}).
\end{equation*}
For $s=i$, recall the definition of $Z_{ij}$, we have
\begin{equation}\label{diagonal}
\int_{\R^N}Z_{sl}\frac{\partial w_{\mathbf{Q}}}{\partial Q_{sj}}=
-\delta_{lj}\int_{\R^N}(\frac{\partial w}{\partial y_j})^2+O(e^{-\rho}).
\end{equation}
For each $(s,l)$, the off-diagonal term gives
\begin{eqnarray}\label{offdiagonal}
&&\sum_{s\neq i}\int_{\R^N}Z_{sl}\frac{\partial (w_{\mathbf{Q}}+\phi_{\mathbf{Q}}) }{\partial Q_{ij}}|_{Q_i=Q_i^\circ}+\sum_{s=i,l\neq j}\int_{\R^N}Z_{sl}\frac{\partial (w_{\mathbf{Q}}+\phi_{\mathbf{Q}})}{\partial Q_{sj} } |_{Q_i=Q_i^\circ}\nonumber\\
&&=(O(e^{-\frac{\rho}{2}})+O(e^{-\rho}))\\
&&=O(e^{-\frac{1}{2}\rho}).\nonumber
\end{eqnarray}

So from (\ref{diagonal}) and (\ref{offdiagonal}), we can see that equation (\ref{e601}) becomes a system of homogeneous equations for $c_{sl}$, and the matrix of the system is nonsingular. So $c_{sl}=0$ for $s=1,\cdots,k, l=1,\cdots,N$. Hence $u_{\mathbf{Q}^\circ}=w_{\mathbf{Q}^\circ}+\phi_{\mathbf{Q}^\circ}$ is a solution of (\ref{p-1}).

Similar to the argument in Section 6 of \cite{LNW}, one can get that $u_{\mathbf{Q}^\circ}>0$ and it has exactly $k$ local maximum points for $\rho $ large enough.

\section{Synchronized vector solutions and the proof of Theorem \ref{theo2}}\label{3}

In this section, we consider the elliptic system (\ref{p-2}) and prove Theorem \ref{theo2}.

\subsection{Notations and Liapunov-Schmidt reduction}\label{3.1}

Let $N \leq 3$ and $w$ be the unique solution of
\begin{equation}
\left\{\begin{array}{l}
\Delta w-w+w^3=0,\\
w(0)=\max_{x\in \R^N} w(x), \ w\to 0 \mbox{ as }|x|\to \infty.
\end{array}
\right.
\end{equation}
It is known that the following asymptotic behavior holds
\begin{equation}
w(r)=A_Nr^{-\frac{N-1}{2}}e^{-r}(1+O(\frac{1}{r})),\
w'(r)=-A_Nr^{-\frac{N-1}{2}}e^{-r}(1+O(\frac{1}{r})),
\end{equation}
for $r$ large, where $A_N>0$ is a constant.

Note that the limit system as $\delta\to 0$ for (\ref{p-2}) is
\begin{equation}\label{limit}
\left\{\begin{array}{l}
-\Delta u+u=\mu_1u^3+\beta v^2u,\\
-\Delta v+v=\mu_2v^3+\beta u^2v,
\end{array}
\right.
\end{equation}
and that
\begin{equation}
\label{UV}
(U,V)=(\alpha w,\gamma w)
\end{equation}
solves (\ref{limit}) provided that $\beta>\max \{\mu_1,\mu_2\}$ or $-\sqrt{\mu_1\mu_2}<\beta<\min \{\mu_1,\mu_2\}$, where
\begin{equation}\label{alpha}
\alpha=\sqrt{\frac{\mu_2-\beta}{\mu_1\mu_2-\beta^2}}, \ \gamma=\sqrt{\frac{\mu_1-\beta}{\mu_1\mu_2-\beta^2}}.
\end{equation}

(It has been proved in \cite{WYO} that for $\beta> \max \{ \mu_1, \mu_2\}$, all solutions to (\ref{limit}) are given by (\ref{UV}).)

We will use $(U,V)$ as the building blocks for the solution of (\ref{p-2}).

Let $\rho>0$  and the configuration space $\Lambda_k$ be defined as in Section \ref{1}. For $\mathbf{Q}_k=(Q_1,\cdots,Q_k)\in \Lambda_k$, we define
\begin{equation}
(U_{Q_i},V_{Q_i})=(U(x-Q_i),V(x-Q_i)),
\end{equation}
and the approximate solution to be
\begin{equation}
U_{\mathbf{Q}_k}=\sum_{i=1}^kU_{Q_i},\
V_{\mathbf{Q}_k}=\sum_{i=1}^kV_{Q_i}.
\end{equation}

Denote by
\begin{equation}
S\vect{u}{v}=\left(\begin{array}{l}
\Delta u-(1+\delta a(x))u+\mu_1u^3+\beta v^2u,\\
\Delta v-(1+\delta b(x))v+\mu_2v^3+\beta u^2v
\end{array}
\right).
\end{equation}

For $f=\vect{f_1}{f_2}, g=\vect{g_1}{g_2}$, we denote by
\begin{equation}
\langle f,g\rangle=\int_{\R^N}f_1g_1+f_2g_2dx.
\end{equation}

The proof of Theorem \ref{theo2} is similar to the proof of  Theorem \ref{theo1}. Fixing $\mathbf{Q}_k=(Q_1,\cdots, Q_k)\in \Lambda_k$, our main idea is to use $(U_{\mathbf{Q}_k},V_{\mathbf{Q}_k})$ as the approximate solution. First using the Liapunov-Schmidt
reduction, we can show that there exists a constant $\rho_0$, such that for $\rho\geq\rho_0$, and $\delta<c_\rho$, for some constant $c_\rho$ depend on $\rho$ but independent of $k$ and $\mathbf{Q}_k$, we can find a $(\phi_{\mathbf{Q}_k},\psi_{\mathbf{Q}_k})$ such that
\begin{equation}
S(\vect{U_{\mathbf{Q}_k}}{V_{\mathbf{Q}_k}}
+\vect{\phi_{\mathbf{Q}_k}}{\psi_{\mathbf{Q}_k}})=\sum_{i=1,\cdots,k,j=1,\cdots,N}c_{ij}\bar{Z}_{ij},
\end{equation}
where $\bar{Z}_{ij}$ is defined as
\begin{equation}\label{zijs}
\bar{Z}_{ij}=\vect{\bar{Z}_{ij,1}}{\bar{Z}_{ij,2}}=\vect{\frac{\partial U_{Q_i}}{\partial x_j}\chi_i(x)}{\frac{\partial V_{Q_i}}{\partial x_j}\chi_i(x)}, \mbox{ for } i=1,\cdots,k, \ j=1,\cdots,N,
\end{equation}
where $\chi_i(x)=\chi(\frac{2|x-Q_i|}{(\rho-1)})$ and $\chi(t)$ is a cut off function , such that $\chi(t)=1$ for $|t|\leq 1$ and $\chi(t)=0$ for $|t|\geq \frac{\rho^2}{\rho^2-1}$. We can show that $(\phi_{\mathbf{Q}_k},\psi_{\mathbf{Q}_k})$ is $C^1$ in $\mathbf{Q}_k$. After that , for any $k$, we define a new function
\begin{equation}
\mathcal{M}(\mathbf{Q}_k)=J(\vect{U_{\mathbf{Q}_k}}{V_{\mathbf{Q}_k}}
+\vect{\phi_{\mathbf{Q}_k}}{\psi_{\mathbf{Q}_k}}),
\end{equation}
we maximize $\mathcal{M}(\mathbf{Q}_k)$ over $\bar{\Lambda}_k$.

\medskip 

For large $\rho$, and fixed points $\mathbf{Q}_k\in \Lambda_k$, we first show solvability in $\{\vect{\phi}{\psi}$, $\{c_{ij}\}\}$  of
the non linear projected problem
\begin{equation}\label{e401s}
\left\{\begin{array}{l}
\Delta (U_{\mathbf{Q}_k}+\phi_{\mathbf{Q}_k})-(1+\delta a(x))(U_{\mathbf{Q}_k}+\phi_{\mathbf{Q}_k})
+\mu_1(U_{\mathbf{Q}_k}+\phi_{\mathbf{Q}_k})^3\\
+\beta (V_{\mathbf{Q}_k}+\psi_{\mathbf{Q}_k})^2(U_{\mathbf{Q}_k}+\phi_{\mathbf{Q}_k})=\sum_{i=1,\cdots,k,j=1,\cdots,N}c_{ij}
\bar{Z}_{ij,1},\\
\Delta (V_{\mathbf{Q}_k}+\psi_{\mathbf{Q}_k})-(1+\delta b(x))(V_{\mathbf{Q}_k}+\psi_{\mathbf{Q}_k})
+\mu_2(V_{\mathbf{Q}_k}+\psi_{\mathbf{Q}_k})^3\\
+\beta (U_{\mathbf{Q}_k}+\phi_{\mathbf{Q}_k})^2(V_{\mathbf{Q}_k}+\psi_{\mathbf{Q}_k})=\sum_{i=1,\cdots,k,j=1,\cdots,N}c_{ij}
\bar{Z}_{ij,2},\\
\langle \vect{\phi_{\mathbf{Q}_k}}{\psi_{\mathbf{Q}_k}},\vect{\bar{Z}_{ij,1}}{\bar{Z}_{ij,2}}\rangle=0\ \mbox{ for }i=1,\cdots,k,j=1,\cdots,N.
\end{array}
\right.
\end{equation}

 \medskip

Define
\begin{equation}\label{W}
W(x)=\sum_{\mathbf{Q}_k\in \Lambda_k}e^{-\eta|x-Q_i|},
\end{equation}
and the norm to be
\begin{equation}
 \quad \| h \|_{**} =\sup_{x \in \R^N} |   W(x)^{-1}  h_1(x) |+\sup_{x \in \R^N} |   W(x)^{-1}  h_2(x) |.
\end{equation}

First we need the following non-degeneracy result:
\begin{lemma}\label{lemma201s}
There exists $\beta^*>0$, such that for $\beta\in (-\beta^*,0)\cup (0, \min\{\mu_1,\mu_2\})\cup (\max\{\mu_1,\mu_2\},\infty)$, $(U,V) $ is non-degenerate for the system (\ref{limit}) in $H^1(\R^N)$ in the sense that the kernel is given by
\begin{equation}
\operatorname{Span}\{(\frac{\partial U}{\partial x_j},\frac{\partial V}{\partial x_j})|j=1,\cdots,N\}.
\end{equation}
\end{lemma}
\begin{proof}
For the proof, see the proof of Proposition 2.3 in \cite{PW}.
\end{proof}

From now on we will always assume that
\begin{equation}
\beta\in (-\beta^*,0)\cup (0, \min\{\mu_1,\mu_2\})\cup (\max\{\mu_1,\mu_2\},\infty).
\end{equation}

Similar as in Section \ref{2}, the following proposition is standard.
\begin{proposition} \label{p401s}
Given $0<\eta<1$. There exist positive numbers $\rho_0$, $C$ and $\xi >0$ such that for all $\rho\geq \rho_0$, and for any $\mathbf{Q}_k\in\Lambda_k$, $\delta <e^{-2\rho}$, there is a unique solution $(\vect{\phi_{\mathbf{Q}_k}}{\psi_{\mathbf{Q}_k}} , \{c_{ij}\} ) $  to problem (\ref{e401s}). Furthermore $(\phi_{\mathbf{Q}_k},\psi_{\mathbf{Q}_k})$ is $C^1$ in $\mathbf{Q}_k$ and we have
\begin{equation}
 \|(\phi_{\mathbf{Q}_k},\psi_{\mathbf{Q}_k}) \|_{**} \leq C\|S\vect{U_{\mathbf{Q}_k}}{V_{\mathbf{Q}_k}}\|_{**}\leq C e^{-\xi  \rho },
\label{est2s}\end{equation}
\begin{equation}
|c_{ij}|\leq Ce^{-\xi  \rho}.
\end{equation}
\end{proposition}

\subsection{A Secondary Liapunov Schmidt Reduction}\label{3.2}
 Similar to the estimate in Section \ref{2.2}, we have the key estimate on the difference between the solutions in the $k-$th step and $ (k+1)-$th step. From now on, we choose $\eta>\frac{1}{2}$.

For $(Q_1,\cdots,Q_{k})\in \Lambda_{k}$,  we  denote $\vect{u_{ Q_1, \cdots, Q_k}}{v_{Q_1,\cdots,Q_k}}$ as $ \vect{U_{ Q_1,..., Q_k}+ \phi_{Q_1, ..., Q_k}}{V_{Q_1,\cdots,Q_k}+\psi_{Q_1,\cdots,Q_k}}$, where $ \vect{\phi_{ Q_1, \cdots, Q_k}}{\psi_{Q_1,\cdots,Q_k}}$ is the unique solution given by Proposition \ref{p401s}.

We now write

\begin{eqnarray}
\vect{u_{ Q_1, \cdots, Q_{k+1}}}{v_{Q_1,\cdots,Q_{k+1}}}&=&\vect{u_{ Q_1, \cdots, Q_k}}{v_{Q_1,\cdots,Q_k}}
+\vect{U_{Q_{k+1}}}{V_{Q_{k+1}}}+\varphi_{k+1}\nonumber\\
&=&\vect{\bar{U}}{\bar{V}}+\vect{\varphi_{k+1,1}}{\varphi_{k+1,2}}
\end{eqnarray}
where
$$\vect{\bar{U}}{\bar{V}}= \vect{u_{ Q_1, \cdots, Q_k}}{v_{Q_1,\cdots,Q_k}}
+\vect{U_{Q_{k+1}}}{V_{Q_{k+1}}}.$$

We have the following estimate for $\varphi_{k+1}$:

 \begin{lemma}\label{lemma501s}
 Let $\rho$, $\delta$ be as in Proposition \ref{p401s}. Then it holds
 \begin{eqnarray}
\label{keyvars}
&&\int_{\R^N} (|\nabla \varphi_{k+1,1}|^2 + \varphi_{k+1,1}^2 ) +(|\nabla \varphi_{k+1,2}|^2 + \varphi_{k+1,2}^2 )dx\\
&&\leq C(e^{-\xi  \rho }\sum_{i=1}^kw(|Q_{k+1}-Q_i|)\nonumber\\
&&+\delta^2(\int_{\R^N}a^2U^2_{Q_{k+1}}+b^2V^2_{Q_{k+1}}dx
+(\int_{\R^N}|a|U_{Q_{k+1}}+|b|V_{Q_{k+1}}dx)^2),\nonumber
 \end{eqnarray}
 for some constant $C>0,\xi>0$ independent of $\rho ,k, \eta$ and $\mathbf{Q}\in \Lambda_{k+1}$.
 \end{lemma}
\begin{proof}
To prove (\ref{keyvars}), we need to perform a further decomposition.

From the non-degeneracy result of $(U,V)$, we have that there are finite many positive eigenvalues to the following
linearized operator:
\begin{equation}
\left(\begin{array}{c}
\Delta \phi_{j,1}-\phi_{j,1}+3\mu_1U^2\phi_{j,1}+\beta V^2\phi_{j,1}+2\beta UV\phi_{j,2}\\
\Delta \phi_{j,2}-\phi_{j,2}+3\mu_2V^2\phi_{j,2}+\beta U^2\phi_{j,2}+2\beta UV\phi_{j,1}
\end{array}
\right)=\lambda_j\vect{\phi_{j,1}}{\phi_{j,2}}
\end{equation}
and the eigenfunctions $\phi_j$ are exponential decay. Assume that $\lambda_j>0$ for $j=1,\cdots,K$. Denote by $\phi_{ij}=\chi_i\phi_j (x-Q_i)$, where $ \chi_i$ is the cut-off function introduced in Section 2 and $\phi_j=\vect{\phi_{j,1}}{\phi_{j,2}}$.

By the equations satisfied by $\varphi_{k+1}$, we have

\begin{equation}\label{varphis}
\bar{L}\varphi_{k+1}=\bar{S}+\sum_{i=1,\cdots,k+1, j=1,\cdots,N}c_{ij}\bar{Z}_{ij}
\end{equation}
for some constants $\{c_{ij}\}$, where
\begin{eqnarray*}
&&\bar{L}\vect{\varphi_{k+1,1}}{\varphi_{k+1,2}}\\
&&=\left(\begin{array}{l}
\Delta \varphi_{k+1,1}-(1+\delta a)\varphi_{k+1,1}+3\mu_1\tilde{U}^2\varphi_{k+1,1}+2\beta \tilde{V}\bar{U}\varphi_{k+1,2}+\beta(\bar{V}+\varphi_{k+1,2})^2\varphi_{k+1,1}\\
\Delta \varphi_{k+1,2}-(1+\delta b)\varphi_{k+1,2}+3\mu_2\hat{V}^2\varphi_{k+1,2}+2\beta \hat{U}\bar{V}\varphi_{k+1,1}+\beta(\bar{U}+\varphi_{k+1,1})^2\varphi_{k+1,2}
\end{array}
\right),
\end{eqnarray*}
and
\begin{eqnarray*}
&&3\tilde{U}^2=\left\{\begin{array}{l}
\frac{(\bar{U}+\varphi_{k+1,1})^3-\bar{U}^3}{ \varphi_{k+1,1}}, \ \mbox{if} \ \varphi_{k+1,1} \not =0\\
3\bar{U}^2, \ \mbox{if} \ \varphi_{k+1,1}=0,
\end{array}
\right.\\
&&2\tilde{V}=\left\{\begin{array}{l}
\frac{(\bar{V}+\varphi_{k+1,2})^2-\bar{V}^2}{ \varphi_{k+1,2}}, \ \mbox{if} \ \varphi_{k+1,2} \not =0\\
2\bar{V}, \ \mbox{if} \ \varphi_{k+1,2}=0,
\end{array}
\right.\\
&&3\hat{V}^2=\left\{\begin{array}{l}
\frac{(\bar{V}+\varphi_{k+1,2})^3-\bar{V}^3}{ \varphi_{k+1,2}}, \ \mbox{if} \ \varphi_{k+1,2} \not =0\\
3\bar{V}^2, \ \mbox{if} \ \varphi_{k+1,2}=0,
\end{array}
\right.\\
&&2\hat{U}=\left\{\begin{array}{l}
\frac{(\bar{U}+\varphi_{k+1,1})^2-\bar{U}^2}{ \varphi_{k+1,1}}, \ \mbox{if} \ \varphi_{k+1,1} \not =0\\
2\bar{U}, \ \mbox{if} \ \varphi_{k+1,1}=0,
\end{array}
\right.
\end{eqnarray*}
and
\begin{eqnarray*}
\bar{S}&=&\left(\begin{array}{l}
\mu_1[(\bar{U}^3-u_{Q_1,\cdots,Q_k}^3-U^3_{Q_{k+1}})]+\beta[\bar{V}^2\bar{U}-v_{Q_1,\cdots,Q_k}^2u_{Q_1,\cdots,Q_k
}-V_{Q_{k+1}}^2U_{Q_{k+1}}]\\
\mu_2[(\bar{V}^3-v_{Q_1,\cdots,Q_k}^3-V^3_{Q_{k+1}})]+\beta[\bar{U}^2\bar{V}-u_{Q_1,\cdots,Q_k}^2v_{Q_1,\cdots,Q_k
}-U_{Q_{k+1}}^2V_{Q_{k+1}}]
\end{array}
\right)\\
&-&\delta\vect{aU_{Q_{k+1}}}{bV_{Q_{k+1}}}.
\end{eqnarray*}

The $L^2$-norm of $\bar{S}$ is estimated first:

By the estimate in Proposition \ref{p401s}, and recall that $\eta>\frac{1}{2}$ we have the following estimate
\begin{eqnarray*}
&&\int_{\R^N}|\mu_1[(\bar{U}^3-u_{Q_1,\cdots,Q_k}^3-U_{Q_{k+1}}^3)]+\beta[\bar{V}^2\bar{U}-v_{Q_1,\cdots,Q_k}^2u_{Q_1,\cdots,Q_k
}-V_{Q_{k+1}}^2U_{Q_{k+1}}]|^2\\
&&+|\mu_2[(\bar{V}^3-v_{Q_1,\cdots,Q_k}^3-V_{Q_{k+1}}^3)]+\beta[\bar{U}^2\bar{V}-u_{Q_1,\cdots,Q_k}^2v_{Q_1,\cdots,Q_k
}-U_{Q_{k+1}}^2V_{Q_{k+1}}]|^2\\
&&\leq Ce^{-\xi\rho}\sum_{i=1}^kw(|Q_{k+1}-Q_i|),
\end{eqnarray*}
and
\begin{eqnarray*}
&&\int_{\R^N}(\delta a U_{Q_{k+1}})^2+(\delta b V_{Q_{k+1}})^2dx \leq C\delta^2\int_{\R^N}a^2U^2_{Q_{k+1}}+b^2V^2_{Q_{k+1}}dx.
\end{eqnarray*}
 So we have
\begin{equation}\label{ss}
\|\bar{S}\|^2_{L^2(\R^N)}\leq C(e^{-\xi\rho}\sum_{i=1}^kw( |Q_{k+1}-Q_i|)+\delta^2\int_{\R^N}a^2U^2_{Q_{k+1}}+b^2V^2_{Q_{k+1}}dx).
\end{equation}

Decompose $\varphi_{k+1}$ as
\begin{equation}\label{decoms}
\varphi_{k+1}=\Psi+\sum_{i=1,\cdots,k+1,l=1,\cdots,K}\ell_{il}\phi_{il}
+\sum_{i=1,\cdots,k+1,j=1,\cdots,N}d_{ij}\bar{Z}_{ij}
\end{equation}
for some $\ell_{il},d_{ij}$ such that
\begin{equation}
\label{345}
\langle \Psi, \phi_{il}\rangle=\langle\Psi, \bar{Z}_{ij}\rangle=0,\ i=1,..., k+1, \ j=1,..., N,\ l=1,\cdots,K.
\end{equation}

Since
\begin{equation}
\varphi_{k+1}=\vect{\phi_{Q_1,\cdots,Q_{k+1}}}{\psi_{Q_1,\cdots,Q_{k+1}}}
-\vect{\phi_{Q_1,\cdots,Q_k}}{\psi_{Q_1,\cdots,Q_k}},
\end{equation}

we have for $i=1,\cdots,k$,
\begin{eqnarray*}
d_{ij}&=&\langle\varphi_{k+1},\bar{Z}_{ij}\rangle+\sum_{l=1,\cdots,K}\ell_{il}\langle\phi_{il},\bar{Z}_{ij}\rangle\\
&=&\langle \vect{\phi_{Q_1,\cdots,Q_{k+1}}}{\psi_{Q_1,\cdots,Q_{k+1}}}
-\vect{\phi_{Q_1,\cdots,Q_k}}{\psi_{Q_1,\cdots,Q_k}},\bar{Z}_{ij}\rangle
+\sum_{l=1,\cdots,K}\ell_{il}\langle\phi_{il},\bar{Z}_{ij}\rangle\\
&=&\sum_{l=1,\cdots,K}\ell_{il}\langle\phi_{il},\bar{Z}_{ij}\rangle\\
&=&e^{-\xi \rho}\sum_{l=1,\cdots,K}\ell_{il}
\end{eqnarray*}
and
\begin{eqnarray*}
d_{k+1,j}&=&\langle\varphi_{k+1},\bar{Z}_{k+1,j}\rangle+
\sum_{l=1,\cdots,K}\ell_{k+1,l}\langle\phi_{k+1,l},\bar{Z}_{k+1,j}\rangle\\
&=&\langle \vect{\phi_{Q_1,\cdots,Q_{k+1}}}{\psi_{Q_1,\cdots,Q_{k+1}}}
-\vect{\phi_{Q_1,\cdots,Q_k}}{\psi_{Q_1,\cdots,Q_k}},\bar{Z}_{k+1,j}\rangle
+\sum_{l=1,\cdots,K}\ell_{k+1,l}\langle\phi_{k+1,l},\bar{Z}_{k+1,j}\rangle\\
&=&-\langle \vect{\phi_{Q_1,\cdots,Q_k}}{\psi_{Q_1,\cdots,Q_k}},\bar{Z}_{k+1,j}\rangle
+\sum_{l=1,\cdots,K}\ell_{k+1,l}\langle\phi_{k+1,l},\bar{Z}_{k+1,j}\rangle\\
&=&-\langle \vect{\phi_{Q_1,\cdots,Q_k}}{\psi_{Q_1,\cdots,Q_k}},\bar{Z}_{k+1,j}\rangle+e^{-\xi \rho}\sum_{l=1,\cdots,K}\ell_{k+1,l}
\end{eqnarray*}
where we have used the orthogonality of the eigenfunctions and the orthogonality conditions satisfied by $(\phi_{Q_1,\cdots,Q_k},\psi_{Q_1,\cdots,Q_k})$ and $(\phi_{Q_1,\cdots,Q_{k+1}},\psi_{Q_1,\cdots,Q_{k+1}})$.
So by Proposition \ref{p401s}, we have
\begin{equation}\label{ds}
\left\{\begin{array}{l}
|d_{ij}|\leq ce^{-\xi \rho}\sum_{l=1,\cdots,K}\ell_{il} \mbox{ for }i=1,\cdots,k,\\
\\
|d_{k+1,j}|\leq ce^{-\xi\rho }\sum_{i=1}^ke^{-\eta|Q_i-Q_{k+1}|}+e^{-\xi \rho }\sum_{l=1,\cdots,K}\ell_{k+1,l}.
\end{array}
\right.
\end{equation}

By (\ref{decoms}), we can rewrite (\ref{varphis}) as
\begin{equation}\label{decom1s}
\bar{L}\Psi+\sum_{i=1,\cdots,k+1,l=1,\cdots,K} \ell_{il}\bar{L}\phi_{il}+\sum_{i=1,\cdots,k+1,j=1,\cdots,N}d_{ij}\bar{L}\bar{Z}_{ij}
=\bar{S}+\sum_{i=1,\cdots,k+1,j=1,\cdots,N}c_{ij}\bar{Z}_{ij}.
\end{equation}

To obtain the estimates for the coefficients $\ell_{il}$ , we use the equation (\ref{decom1s}).

First, multiplying (\ref{decom1s}) by $\phi_{tl}$ and integrating over $\R^N$, we have
\begin{eqnarray}
\label{W345s}
\ell_{il}\langle\bar{L}(\phi_{il}),\phi_{il}\rangle&=&-\sum_{j=1}^N d_{ij}\langle \bar{L}(\bar{Z}_{ij}),\phi_{il}\rangle+\langle\bar{S},\phi_{il}\rangle\\
&+&\sum_{j\neq l}\ell_{ij}\langle\bar{L}(\phi_{ij}),\phi_{il}\rangle\nonumber
\end{eqnarray}
where
\begin{equation}
\left\{\begin{array}{ll}
\label{W346s}
|\langle\bar{S},\phi_{il}\rangle|\leq ce^{-\xi \rho} e^{-\eta|Q_i-Q_{k+1}|}+\delta |\langle\vect{aU_{Q_{k+1}}}{bV_{Q_{k+1}}},\phi_{il} \rangle| \mbox{ for }i=1,\cdots,k\\
\\
|\langle\bar{S},\phi_{k+1,l}\rangle|\leq ce^{-\xi\rho}\sum_{i=1}^k e^{-\eta|Q_i-Q_{k+1}|}+\delta |\langle\vect{aU_{Q_{k+1}}}{bV_{Q_{k+1}}},\phi_{k+1,l} \rangle|.
\end{array}
\right.
\end{equation}

By the equation satisfied by $\phi_l$, we have
\begin{equation}
\label{W234s}
\langle\bar{L}\phi_{ij},\phi_{il}\rangle =  - \delta_{jl}\lambda_l\lambda_j \langle \phi_l,\phi_j\rangle + O(e^{-\xi \, \rho}).
\end{equation}

Combining (\ref{ds}) and (\ref{W345s})-(\ref{W234s}), we have
\begin{equation}\label{cs}
\left\{\begin{array}{ll}
|\ell_{il}|\leq  Ce^{-\xi \rho} e^{-\eta|Q_i-Q_{k+1}|}\\
+\sum_{j=1,\cdots,N}\delta |\langle\vect{aU_{Q_{k+1}}}{bV_{Q_{k+1}}},\phi_{ij} \rangle|+e^{-\xi\rho}\|\Psi\|_{H^1(B_{\frac{\rho}{2}}(Q_i))}, \ i=1,..., k\\
\\
|\ell_{k+1,l}|\leq  Ce^{-\xi \rho}\sum_{i=1}^k e^{-\eta|Q_i-Q_{k+1}|}\\
+\sum_{j=1,\cdots,N}\delta|\langle\vect{aU_{Q_{k+1}}}{bV_{Q_{k+1}}},\phi_{k+1,j} \rangle|+e^{-\xi\rho}\|\Psi\|_{H^1(B_{\frac{\rho}{2}}(Q_{k+1}))},
\end{array}
\right.
\end{equation}
and
\begin{equation}\label{d1s}
\left\{\begin{array}{l}
|d_{ij}|\leq  ce^{-\xi \rho} e^{-\eta|Q_i-Q_{k+1}|}\\
+\sum_{j=1,\cdots,N}\delta|\langle\vect{aU_{Q_{k+1}}}{bV_{Q_{k+1}}},\phi_{ij} \rangle|
 +e^{-\xi\rho}\|\Psi\|_{H^1(B_{\frac{\rho}{2}}(Q_i))}\mbox{ for }i=1,\cdots,k,\\
 \\
|d_{k+1,j}|\leq ce^{-\xi\rho }\sum_{i=1}^ke^{-\eta|Q_i-Q_{k+1}|}\\
+\sum_{j=1,\cdots,N}\delta |\langle\vect{aU_{Q_{k+1}}}{bV_{Q_{k+1}}},\phi_{k+1,j} \rangle|
+e^{-\xi\rho}\|\Psi\|_{H^1(B_{\frac{\rho}{2}}(Q_{k+1}))}.
\end{array}
\right.
\end{equation}

Next let us estimate $\Psi$. Multiplying (\ref{decom1s}) by $\Psi$ and integrating over $\R^N$, we find
\begin{eqnarray}\label{psis}
\langle \bar{L}(\Psi),\Psi\rangle&=&\langle\bar{S} ,\Psi\rangle-\sum_{i=1,\cdots,k+1,j=1,\cdots,N}d_{ij}\langle \bar{L}(\bar{Z}_{ij}),\Psi\rangle\\
&-&\sum_{i=1,\cdots,k+1,l=1,\cdots,K}\ell_{il}\langle\bar{L}\phi_{il},\Psi\rangle.\nonumber
\end{eqnarray}
We claim that
\begin{equation}
-\langle\bar{L}(\Psi),\Psi\rangle\geq c_0 \|\Psi\|^2_{H^1(\R^N)}
\end{equation}
for some constant $c_0>0$.

Since the approximate solution is exponentially decay away from the points $Q_i$, we have
\begin{equation}
\langle\bar{L}(\Psi),\Psi\rangle_{\R^N \backslash  \cup_i B_{\frac{\rho-1}{2}}(Q_i)}
\geq \frac{1}{2}
\int_{\R^N \backslash \cup_i B_{\frac{\rho-1}{2}}(Q_i)}|\nabla \Psi_1|^2+|\Psi_1|^2+|\nabla \Psi_2|^2+|\Psi_2|^2dx.
\end{equation}
Now we only need to prove the above estimates in the domain $\cup_i B_{\frac{\rho-1}{2}}(Q_i)$. We prove it by contradiction. Otherwise, there exists a sequence $\rho_n\to +\infty$, and $Q_i^{(n)}$ such that
\begin{eqnarray*}
&&\int_{B_{\frac{\rho_n-1}{2}}(Q_i^{(n)})}|\nabla \Psi^{(n)}_1|^2+|\Psi^{(n)}_1|^2+|\nabla \Psi^{(n)}_2|^2+|\Psi^{(n)}_2|^2dx=1,\\ &&\langle\bar{L}(\Psi^{(n)}),\Psi^{(n)}\rangle_{B_{\frac{\rho_n-1}{2}}(Q_i^{(n)})}\to 0,\mbox{ as } n\to \infty.
\end{eqnarray*}
Then we can extract from the sequence $\Psi^{(n)}(\cdot-Q_i^{(n)})$ a subsequence which will converge weakly in $H^1(\R^n)$ to $\Psi_\infty$, such that
\begin{eqnarray}\label{phi1s}
&&\int_{\R^N}|\nabla \Psi_{\infty,1}|^2+|\Psi_{\infty,1}|^2-3\mu_1U^2\Psi_{\infty,1}^2+|\nabla \Psi_{\infty,2}|^2+|\Psi_{\infty,2}|^2-3\mu_2V^2\Psi_{\infty,2}^2\nonumber\\
&&-\beta U^2\Psi_{\infty,2}^2-\beta V^2\Psi_{\infty,1}^2-4\beta UV \Psi_{\infty,1}\Psi_{\infty,2}dx=0,
\end{eqnarray}
and
\begin{equation}\label{phi2s}
\langle\Psi_\infty, \phi_i\rangle=\langle\Psi_\infty ,\vect{\partial_{x_j}U}{\partial_{x_j}V}\rangle=0, \mbox{ for }i=1,\cdots,K, \ j=1,\cdots,N.
\end{equation}
From (\ref{phi1s}) and (\ref{phi2s}), we deduce that $\Psi_\infty=0$.

Hence
\begin{equation}
\Psi^{(n)}\rightharpoonup 0 \mbox{ weakly } \mbox{ in } H^1(\R^N).
\end{equation}
So

\begin{eqnarray*}
&&\int_{\R^N}3\mu_1\tilde{U}^2(\Psi^{(n)}_1)^2+2\beta \tilde{V}\bar{U}\Psi^{(n)}_2\Psi^{(n)}_1\\
&&+\beta(\bar{V}+\varphi_{k+1,2})^2(\Psi^{(n)}_1)^2+3\mu_2\hat{V}^2(\Psi^{(n)}_2)^2\\
&&+2\beta \hat{U}\bar{V}\Psi^{(n)}_1\Psi^{(n)}_2+\beta(\bar{U}+\varphi_{k+1,1})^2(\Psi^{(n)}_2)^2dx\to 0  \mbox{ as }n\to \infty.
\end{eqnarray*}

We have
\begin{equation}
\|\Psi^{(n)}\|_{H^1(B_{\frac{\rho_n-1}{2}})}\to 0 \mbox{ as }n\to \infty.
\end{equation}
This contradicts the assumption
\begin{equation}
\|\Psi^{(n)}\|_{H^1}=1.
\end{equation}

So we get that
\begin{equation}\label{psi1s}
-\langle \bar{L}(\Psi),\Psi \rangle \geq c_0 \|\Psi\|^2_{H^1(\R^N)}.
\end{equation}

From (\ref{psis}) and (\ref{psi1s}), we get
\begin{eqnarray}\label{Psis}
&&\|\Psi\|^2_{H^1 (\R^N) }\leq c(\sum_{ij}|d_{ij} | |\langle\bar{L}(\bar{Z}_{ij}), \Psi \rangle| +\sum_{il}|\ell_{il}||\langle\bar{L}\phi_{il},\Psi\rangle|+|\langle \bar{S}, \Psi\rangle | ) \\
&&\leq c(\sum_{ij}|d_{ij}|\|\Psi\|_{H^1 (B_{\frac{\rho}{2}}(Q_i)) }+\sum_{il} |\ell_{il}|\|\Psi\|_{H^1(B_{\frac{\rho}{2}}(Q_i)) }+\|\bar{S}\|_{L^2 (\R^N) }\|\Psi\|_{H^1 (\R^N)}).\nonumber
\end{eqnarray}

From (\ref{cs}) (\ref{d1s}) (\ref{ss}) and (\ref{Psis}), we choose $\eta>\frac{1}{2}$,  we get that
\begin{eqnarray}\label{varphi1s}
\|\varphi_{k+1}\|_{H^1 (\R^N )}&\leq& C(e^{-\xi\rho}\sum_{i=1}^ke^{-\eta |Q_{k+1}-Q_i|}+\delta \int_{\R^n}|a|U_{Q_{k+1}}+|b|V_{Q_{k+1}}dx+\|\bar{S}\|_{L^2})\nonumber\\
&\leq& C(e^{-\xi\rho}\sum_{i=1}^ke^{-\eta|Q_{k+1}-Q_i|}+e^{-\xi  \rho }(\sum_{i=1}^kw(|Q_{k+1}-Q_i|))^{\frac{1}{2}}\\
&+&\delta\int_{\R^N}|a|U_{Q_{k+1}}+|b|V_{Q_{k+1}}dx
+\delta(\int_{\R^N}a^2U^2_{Q_{k+1}}+b^2V^2_{Q_{k+1}}dx)^{\frac{1}{2}}),\nonumber
\end{eqnarray}

Since we choose $\eta>\frac{1}{2}$, we have
\begin{equation}\label{ineqs}
(\sum_{i=1}^ke^{-\eta|Q_i-Q_{k+1}|})^2\leq C\sum_{i=1}^kw(|Q_i-Q_{k+1}|).
\end{equation}

By (\ref{varphi1s}) and (\ref{ineqs}), we thus get that
\begin{eqnarray}
\|\varphi_{k+1}\|_{H^1 (\R^N )}
&\leq& C(e^{-\xi  \rho }(\sum_{i=1}^kw(|Q_{k+1}-Q_i|))^{\frac{1}{2}}
+\delta\int_{\R^N}|a|U_{Q_{k+1}}+|b|V_{Q_{k+1}}dx\nonumber\\
&+&\delta(\int_{\R^N}a^2U^2_{Q_{k+1}}+b^2V^2_{Q_{k+1}}dx)^{\frac{1}{2}}).
\end{eqnarray}

Moreover, from the estimate (\ref{cs}) and (\ref{d1s}), and take into consideration that $\chi_i$ is supported in $B_{\frac{\rho}{2}}(Q_i)$, using holder inequality, we can get a more accurate estimate on $\varphi_{k+1}$,
\begin{eqnarray}\label{varphi2s}
\|\varphi_{k+1}\|_{H^1 (\R^n )}
&\leq& C(e^{-\xi  \rho }(\sum_{i=1}^kw(|Q_{k+1}-Q_i|))^{\frac{1}{2}}\nonumber\\
&+&\delta\sum_{i=1,\cdots,k+1}(\int_{B_{\frac{\rho}{2}}(Q_i)}a^2U^2_{Q_{k+1}}+b^2V^2_{Q_{k+1}}dx)^{\frac{1}{2}}\nonumber\\
&+&\delta(\int_{\R^n}a^2U^2_{Q_{k+1}}+b^2V^2_{Q_{k+1}}dx)^{\frac{1}{2}}).
\end{eqnarray}

\end{proof}

\subsection{The Reduced Problem: A Maximization Procedure}\label{3.3}
\setcounter{equation}{0}
In this section, we study a maximization problem. Fix $\mathbf{Q}_k\in \Lambda_k$, we define a new functional
\begin{equation}
\mathcal{M}(\mathbf{Q}_k)=J(u_{\mathbf{Q}_k},v_{\mathbf{Q}_k}): \Lambda_k \rightarrow \R.
\end{equation}
Define
\begin{equation}
\mathcal{C}_k=\max_{\mathbf{Q}\in\Lambda_k}\{\mathcal{M}(\mathbf{Q}_k)\}.
\end{equation}
Since $\mathcal{M}(\mathbf{Q}_k)$ is continuous in $\mathbf{Q}_k$, the maximization problem has a solution. We will show below that the maximization problem has a solution.

We first prove that the maximum can be attained at finite points for each $\mathcal{C}_k$.

\begin{lemma}\label{lemma601s}
Let assumptions $(H'_1),(H'_2)$ and the assumptions in Proposition \ref{p401s} be satisfied. Then, for all $k$:
\begin{itemize}
\item
There exists $\mathbf{Q}_k=(Q_1,Q_2,\cdots,Q_k)\in \Lambda_k$ such that
\begin{equation}
\mathcal{C}_k=\mathcal{M}(\mathbf{Q}_k);
\end{equation}
\item
There holds
\begin{equation}\label{e501s}
\mathcal{C}_{k+1}>\mathcal{C}_k+I(U,V),
\end{equation}
where $I(U,V)$ is the energy of $(U,V)$,
\begin{eqnarray}
I(U,V)&=&\frac{1}{2}\int_{\R^n}|\nabla U|^2+U^2+|\nabla V|^2+V^2dx\nonumber\\
&-&\frac{1}{4}\int_{\R^n}\mu_1 U^4+\mu_2 V^4dx-\frac{\beta}{2}\int_{\R^n}U^2V^2dx
\end{eqnarray}
\end{itemize}
\end{lemma}

\begin{proof}
The proof is similar to the proof of Lemma \ref{lemma601}. We divide the proof into several steps.

{\bf Step 1:}
$\mathcal{C}_1>I(U,V)$, and $\mathcal{C}_1$ can be attained at finite point. First using standard Liapunov-Schmidt reduction, we have
\begin{equation}
\|(\phi_Q,\psi_Q\|_{H^1}\leq C\|\delta (aU_Q,bV_Q)\|_{L^2}.
\end{equation}

Assume that $|Q|\to \infty$, then
\begin{eqnarray*}
&&J(u_Q,v_Q)=\frac{1}{2}\int_{\R^N}|\nabla (U_Q+\phi_Q)|^2+(1+\delta a)(U_Q+\phi_Q)^2dx\\
&&+\frac{1}{2}\int_{\R^N}|\nabla(V_Q+\psi_Q)|^2+(1+\delta b)(V_Q+\psi_Q)^2dx\\
&&-\frac{1}{4}\int_{\R^N}\mu_1(U_Q+\phi_Q)^4+\mu_2(V_Q+\psi_Q)^4dx-\frac{\beta}{2}
\int_{\R^N}(U_Q+\phi_Q)^2(V_Q+\psi_Q)^2dx\\
&&=I(U_Q,V_Q)+\frac{\delta}{2}\int_{\R^N}aU^2_Q+bV^2_Qdx+\delta^2\|(aU_Q,bV_Q)\|^2_{L^2(\R^N)}\\
&&\geq I(U_Q,V_Q)\\
&&+\frac{1}{4}[\int_{B_{\frac{\rho}{2}}(Q)}\delta (aU_Q^2+bV_Q^2)dx -\sup_{B_{\frac{|Q|}{4}}(0)}(U_Q^2+V_Q^2)\int_{supp (\alpha^2a+\gamma^2 b)^-}\delta(|a|+|b|)dx]\\
&&\geq I(U,V)+\frac{1}{4}\int_{B_{\frac{\rho}{2}}(Q)}\delta (aU_Q^2+bV_Q^2) dx-O(e^{-\frac{3}{2} |Q|}).
\end{eqnarray*}

By the slow decay assumption on the potential $a,b$, we get that
\begin{equation*}
\frac{1}{2}\int_{B_{\frac{\rho}{2}}(Q)}\delta (aU_Q^2+bV_Q^2) dx-O(e^{-\frac{3}{2} |Q|})>0.
\end{equation*}
So
\begin{equation*}
\mathcal{C}_1\geq J(u_Q,v_Q)>I(U,V).
\end{equation*}

Let us prove now that $\mathcal{C}_1$ can be attained at finite point. Let $\{Q_i\}$ be a sequence such that
$\lim_{i\to \infty}\mathcal{M}(Q_i)=\mathcal{C}_1$, and assume that $|Q_i|\to \infty$, by the same argument as above,
\begin{eqnarray*}
J(u_{Q_i},v_{Q_i})&=&I(U,V)+\frac{1}{2}\delta \int_{\R^N}aU_{Q_i}^2+bV_{Q_i}^2dx\\
&+&O(\delta^2\int_{\R^N}a^2U_{Q_i}^2+b^2V_{Q_i}^2dx),
\end{eqnarray*}
as $Q_i\to \infty$, by the decay assumption on $a,b $, we have
\begin{eqnarray*}
\frac{\delta}{2} \int_{\R^N}aU_{Q_i}^2+bV_{Q_i}^2dx+O(\delta^2\int_{\R^N}a^2U_{Q_i}^2+b^2V_{Q_i}^2dx)\to 0\mbox{ as }i\to \infty.
\end{eqnarray*}
Thus,
\begin{equation}
\mathcal{C}_1=\lim_{i \to \infty}J(u_{Q_i},v_{Q_i})\leq I(U,V) \mbox{ as }i\to \infty.
\end{equation}
Contradiction! So $\mathcal{C}_1$ can be attained at finite point.

{\bf Step 2:}
Assume that there exists $\mathbf{Q}_k=(\bar{Q}_1,\cdots,\bar{Q}_k)\in \Lambda_k$ such that $\mathcal{C}_k=\mathcal{M}(\mathbf{Q}_k)$, and we denote the solution by $(u_{\bar{Q}_1,\cdots, \bar{Q}_k},v_{\bar{Q}_1,\cdots,\bar{Q}_k})$.

Next, we prove that there exists $(Q_1,\cdots,Q_{k+1})\in \Lambda_{k+1}$ such that $\mathcal{C}_{k+1}$ can be attained.

Let $((Q_1^{(n)},\cdots,Q_{k+1}^{(n)}))_n$ be a sequence such that
\begin{equation}\label{ck1s}
\mathcal{C}_{k+1}=\lim_{n\to \infty }\mathcal{M}(Q_1^{(n)},\cdots,Q_{k+1}^{(n)}).
\end{equation}

We claim that $(Q_1^{(n)},\cdots,Q_{k+1}^{(n)})$ is bounded. We prove it by contradiction. Wlog, we assume that $|Q_{k+1}^{(n)}|\to \infty$ as $n\to \infty$. In the following we omit the index $n$ for simplicity.

\begin{eqnarray}\label{eq601s}
&&J(u_{Q_1,\cdots,Q_{k+1}},v_{Q_1,\cdots,Q_{k+1}})\\
&&=J(\vect{u_{Q_1,\cdots,Q_k}}{v_{Q_1,\cdots,Q_k}}+\vect{U_{Q_{k+1}}}{V_{Q_{k+1}}}+\vect{\varphi_{k+1,1}}{\varphi_{k+1,2}})\nonumber\\
&&=J(\vect{u_{Q_1,\cdots,Q_k}}{v_{Q_1,\cdots,Q_k}}+\vect{U_{Q_{k+1}}}{V_{Q_{k+1}}})\nonumber\\
&&+\|\bar{S}(u_{Q_1,\cdots,Q_k}+U_{Q_{k+1}},v_{Q_1,\cdots,Q_k}+V_{Q_{k+1}})
\|_{L^2}\|\varphi_{k+1}\|_{H^1}\nonumber
\end{eqnarray}
\begin{eqnarray}
&&+\sum_{i=1,\cdots,k,j=1,\cdots,N}c_{ij}\langle \bar{Z}_{ij},\varphi_{k+1}\rangle+O(\|\varphi_{k+1}\|_{H^1}^2)\nonumber\\
&&=J(\vect{u_{Q_1,\cdots,Q_k}}{v_{Q_1,\cdots,Q_k}}+\vect{U_{Q_{k+1}}}{V_{Q_{k+1}}})\nonumber\\
&&+O(e^{-\xi\rho}\sum_{i=1}^kw( |Q_{k+1}- Q_i|)+\delta^2\int_{\R^N}a^2U_{Q_{k+1}}^2+b^2V_{Q_{k+1}}^2dx\nonumber\\
&&+\delta^2(\int_{\R^N}|a|U_{Q_{k+1}}+|b|V_{Q_{k+1}}dx)^2,\nonumber
\end{eqnarray}
where we use the condition that $\langle \bar{Z}_{ij},\varphi_{k+1}\rangle=0$ for $i=1,\cdots,k$,
and

\begin{eqnarray}\label{eq602s}
&&J(\vect{u_{Q_1,\cdots,Q_k}}{v_{Q_1,\cdots,Q_k}}+\vect{U_{Q_{k+1}}}{V_{Q_{k+1}}})\\
&&\leq \mathcal{C}_k+I(U_{Q_{k+1}},V_{Q_{k+1}})+\frac{\delta}{2}\int_{\R^N}aU_{Q_{k_1}}^2+bV_{Q_{k+1}}^2dx\nonumber\\
&&+\sum_{i=1,\cdots,k,j=1,\cdots,N}c_{ij}\langle \bar{Z}_{ij},\vect{U_{Q_{k+1}}}{V_{Q_{k+1}}}\rangle\nonumber\\
&&-\int_{\R^N}\mu_1U_{Q_{k+1}}^3u_{Q_1,\cdots,Q_k}+\mu_2V_{Q_{k+1}}^3v_{Q_1,\cdots,Q_k}dx\nonumber\\
&&-\beta\int_{\R^N}U_{Q_{k+1}}^2V_{Q_{k+1}}v_{Q_1,\cdots,Q_k}+V_{Q_{k+1}}^2U_{Q_{k+1}}u_{Q_1,\cdots,Q_k}dx\nonumber\\
&&+O(e^{-\xi\rho}\sum_{i=1}^kw( |Q_{k+1}- Q_i|)).\nonumber
\end{eqnarray}

By estimates (\ref{est2s}) and that the definition of $\bar{Z}_{ij}$,  we have
\begin{eqnarray}\label{eq603s}
|\sum_{i=1}^kc_{ij}\langle \bar{Z}_{ij},\vect{U_{Q_{k+1}}}{V_{Q_{k+1}}}\rangle|\leq ce^{-\xi \rho}\sum_{i=1}^kw(|Q_{k+1}-Q_i|).
\end{eqnarray}

Similar as the estimate (\ref{eq6031}), we can get that
\begin{eqnarray*}
&&\int_{\R^N}\mu_1U_{Q_{k+1}}^3\phi_{Q_1,\cdots,Q_k}+\mu_2V_{Q_{k+1}}^3\psi_{Q_1,\cdots,Q_k}dx\\
&&-\beta\int_{\R^N}U_{Q_{k+1}}^2V_{Q_{k+1}}\psi_{Q_1,\cdots,Q_k}+V_{Q_{k+1}}^2U_{Q_{k+1}}
\phi_{Q_1,\cdots,Q_k}dx\\
\end{eqnarray*}
\begin{eqnarray*}
&&\leq C(\delta e^{-\xi \rho}\int_{\R^N}\sum_{i=1}^ke^{-\eta|x-Q_i|}(aU_{Q_{k+1}}+bV_{Q_{k+1}})dx\\
&&+\delta\int_{\R^n}aU_{\mathbf{Q}_k}U_{Q_{k+1}}+bV_{\mathbf{Q}_k}V_{Q_{k+1}}dx+e^{-\xi \rho }\sum_{i=1}^kw( |Q_{k+1}-Q_i|)).
\end{eqnarray*}

So we have

\begin{eqnarray}\label{eq604s}
&&\int_{\R^N}\mu_1U_{Q_{k+1}}^3u_{Q_1,\cdots,Q_k}+\mu_2V_{Q_{k+1}}^3v_{Q_1,\cdots,Q_k}dx\\
&&+\beta\int_{\R^N}U_{Q_{k+1}}^2V_{Q_{k+1}}v_{Q_1,\cdots,Q_k}+V_{Q_{k+1}}^2U_{Q_{k+1}}u_{Q_1,\cdots,Q_k}dx\nonumber\\
&&=\int_{\R^N}\mu_1U_{Q_{k+1}}^3U_{Q_1,\cdots,Q_k}+\mu_2V_{Q_{k+1}}^3V_{Q_1,\cdots,Q_k}dx\nonumber\\
&&+\beta\int_{\R^N}U_{Q_{k+1}}^2V_{Q_{k+1}}V_{Q_1,\cdots,Q_k}+V_{Q_{k+1}}^2U_{Q_{k+1}}U_{Q_1,\cdots,Q_k}dx
\nonumber\\
&&+O(\delta e^{-\xi \rho}\int_{\R^n}\sum_{i=1}^ke^{-\eta|x-Q_i|}(aU_{Q_{k+1}}+bV_{Q_{k+1}})dx\nonumber\\
&&+\delta\int_{\R^N}aU_{\mathbf{Q}_k}U_{Q_{k+1}}
+bV_{\mathbf{Q}_k}V_{Q_{k+1}}dx+e^{-\xi \rho }\sum_{i=1}^kw( |Q_{k+1}-Q_i|))\nonumber\\
&&\geq \frac{1}{4}A\gamma_1 \sum_{i=1}^kw(|Q_{k+1}-Q_i|)+O(e^{-\xi \rho }\sum_{i=1}^kw( |Q_{k+1}-Q_i|))\nonumber\\
&&+O(\delta e^{-\xi \rho}\int_{\R^N}\sum_{i=1}^ke^{-\eta|x-Q_i|}(aU_{Q_{k+1}}+bV_{Q_{k+1}})dx\nonumber\\
&&+\delta\int_{\R^N}aU_{\mathbf{Q}_k}U_{Q_{k+1}}
+bV_{\mathbf{Q}_k}V_{Q_{k+1}}dx)\nonumber
\end{eqnarray}
where $\gamma_1$ is defined in (\ref{gamma1}) with $f(t)=t^3$ and $A=\mu_1\alpha^4+\mu_2\gamma^4+2\beta\alpha^2\gamma^2>0$.

So by (\ref{eq601s}), (\ref{eq602s}), (\ref{eq603s}) and (\ref{eq604s}), we obtain
\begin{eqnarray}\label{eq605s}
&&J(u_{Q_1,\cdots,Q_{k+1}},v_{Q_1,\cdots,Q_{k+1}})\\
&&\leq \mathcal{C}_k+I(U,V)+\frac{\delta}{2}\int_{\R^n} a U^2_{Q_{k+1}}+b V_{Q_{k+1}}^2dx\nonumber\\
&&-\frac{1}{4}A\gamma_1 \sum_{i=1}^kw(|Q_{k+1}-Q_i|)+O(e^{-\xi \rho }\sum_{i=1}^kw( |Q_{k+1}-Q_i|))\nonumber
\end{eqnarray}
\begin{eqnarray}
&&+O(\delta e^{-\xi \rho}\int_{\R^N}\sum_{i=1}^ke^{-\eta|x-Q_i|}(aU_{Q_{k+1}}+bV_{Q_{k+1}})dx+\delta\int_{\R^N}aU_{\mathbf{Q}_k}U_{Q_{k+1}}
+bV_{\mathbf{Q}_k}V_{Q_{k+1}}dx)\nonumber\\
&&+\delta^2\int_{\R^N}a^2U_{Q_{k+1}}^2+b^2V_{Q_{k+1}}^2
+\delta^2(\int_{\R^N}|a|U_{Q_{k+1}}+|b|V_{Q_{k+1}})^2dx.\nonumber
\end{eqnarray}

By the assumption that $|Q_{k+1}^{(n)}|\to \infty$,
\begin{eqnarray*}
&& \delta e^{-\xi\rho}\int_{\R^N}\sum_{i=1}^ke^{-\eta|x-Q_i|}(aU_{Q_{k+1}}+bV_{Q_{k+1}})dx\\
&&+\delta \int_{\R^N}aU_{\mathbf{Q}_k}U_{Q_{k+1}}+bV_{\mathbf{Q}_k}V_{Q_{k+1}}dx\\
&&+\delta^2\int_{\R^N}a^2U_{Q_{k+1}}^2+b^2V_{Q_{k+1}}^2
+\delta^2(\int_{\R^N}|a|U_{Q_{k+1}}+|b|V_{Q_{k+1}})^2
\to 0 \mbox{ as }n \to \infty ,
\end{eqnarray*}
and
 \begin{equation}
 -\frac{1}{4}A\gamma_1 \sum_{i=1}^kw(|Q_{k+1}-Q_i|)+O(e^{-\xi \rho}\sum_{i=1}^kw( |Q_{k+1}-Q_i|))<0.
  \end{equation}

Combining (\ref{ck1s}), (\ref{eq605s}) and the above estimates, we have
\begin{equation}\label{leqs}
\mathcal{C}_{k+1}\leq \mathcal{C}_k+I(U,V).
\end{equation}

On the other hand, since by the assumption, $\mathcal{C}_k$ can be attained at $(\bar{Q}_1,\cdots,\bar{Q}_k)$, so there exists other point $Q_{k+1}$ which is far away from the $k$ points which will be determined later. Next let us consider the solution concentrated at the points $(\bar{Q}_1,\cdots, \bar{Q_{k}}, Q_{k+1})$. We denote the solution by $(u_{\bar{Q}_1,\cdots,\bar{Q}_k,Q_{k+1}},v_{\bar{Q}_1,\cdots,\bar{Q}_k,Q_{k+1}})$. By similar argument as the above, using the estimate (\ref{varphi2s}) instead of (\ref{keyvars}), we have the following estimates:
\begin{eqnarray}
&&J(u_{\bar{Q}_1,\cdots,\bar{Q}_k, Q_{k+1}},v_{\bar{Q}_1,\cdots,\bar{Q}_k,Q_{k+1}})\\
&&=J(u_{\bar{Q}_1,\cdots,\bar{Q}_k},v_{\bar{Q}_1,\cdots,\bar{Q}_k})+I(U,V)\nonumber\\
&&+\frac{\delta}{2}\int_{\R^N}aU^2_{Q_{k+1}}+bV^2_{Q_{k+1}}dx+O(\sum_{i=1}^kw( |Q_{k+1}-\bar{Q}_i|))\nonumber\\
&&+O(\delta e^{-\xi \rho}\int_{\R^N}\sum_{i=1}^ke^{-\eta|x-Q_i|}(aU_{Q_{k+1}}+bV_{Q_{k+1}})dx
+\delta \int_{\R^N}a U_{\bar{\mathbf{Q}}_k}U_{Q_{k+1}}+b V_{\bar{\mathbf{Q}}_k}V_{Q_{k+1}}dx\nonumber\\
&&+\delta^2 (\int_{\R^N}a^2U_{Q_{k+1}}^2+b^2 V_{Q_{k+1}}^2 dx+(\sum_{i=1,\cdots,k+1}(\int_{B_{\frac{\rho}{2}}(Q_i)}a^2U^2_{Q_{k+1}}
+b^2V^2_{Q_{k+1}}dx)^{\frac{1}{2}})^2)).\nonumber
\end{eqnarray}

By the slow decay assumption of $a,b$ at infinity, i.e. $\lim_{|x|\to \infty}(\alpha^2a+\gamma^2b)e^{\bar{\eta}|x|}=+\infty$ as $|x|\to \infty$, for some $\bar{\eta}<1$, we can further choose $\eta>\bar{\eta}$, and choose $Q_{k+1}$ such that
\begin{equation}
|Q_{k+1}|\geq \frac{\max_{i=1}^k|\bar{Q}_i|+\ln \delta}{\eta-\bar{\eta}}.
\end{equation}
This implies that
\begin{eqnarray}
&&\frac{\delta}{2}\int_{\R^n}aU^2_{Q_{k+1}}+bV^2_{Q_{k+1}}dx+O(\sum_{i=1}^kw( |Q_{k+1}-\bar{Q}_i|))\\
&&+O(\delta e^{-\xi \rho}\int_{\R^n}\sum_{i=1}^ke^{-\eta|x-Q_i|}(aU_{Q_{k+1}}+bV_{Q_{k+1}})dx
+\delta \int_{\R^n}a U_{\bar{\mathbf{Q}}_k}U_{Q_{k+1}}+b V_{\bar{\mathbf{Q}}_k}V_{Q_{k+1}}dx\nonumber\\
&&+\delta^2 (\int_{\R^n}a^2U_{Q_{k+1}}^2+b^2 V_{Q_{k+1}}^2 dx+(\sum_{i=1,\cdots,k+1}(\int_{B_{\frac{\rho}{2}}(Q_i)}a^2U^2_{Q_{k+1}}+b^2V^2_{Q_{k+1}}dx)^{\frac{1}{2}})^2))\nonumber\\
&&\geq Ce^{-\bar{\eta}|Q_{k+1}|}-O(\sum_{i=1}^ke^{-\eta|Q_i-Q_{k+1}|})>0\nonumber,
\end{eqnarray}
so
\begin{equation}\label{geqs}
\mathcal{C}_{k+1}\geq J(u_{\bar{Q}_1,\cdots,\bar{Q}_k, Q_{k+1}},v_{\bar{Q}_1,\cdots,\bar{Q}_k, Q_{k+1}})>\mathcal{C}_k+I(U,V).
\end{equation}
Combining (\ref{geqs}) and (\ref{leqs}), one get that
\begin{equation}
\mathcal{C}_k+I(U,V)<\mathcal{C}_{k+1}\leq \mathcal{C}_k+I(U,V).
\end{equation}
We have reached a  contradiction with (\ref{leqs}). So we get that $\mathcal{C}_{k+1}$ can be attained at finite points in $\Lambda_{k+1}$.

Moreover, from the proof above, we can get a relation between $\mathcal{C}_{k+1}$ and $\mathcal{C}_k$:
\begin{equation}
\mathcal{C}_{k+1}\geq \mathcal{C}_k+I(U,V).
\end{equation}
\end{proof}

The next following Proposition excludes boundary maximization.

\begin{proposition}\label{p602s}
The maximization problem
\begin{equation}
\max_{\mathbf{Q}_k\in \bar{\Lambda}_k} \mathcal{M}(\mathbf{Q}_k)
\end{equation}
has a solution $\mathbf{Q}_k \in \Lambda_k^\circ$, i.e., the interior of $\Lambda_k$.
\end{proposition}
\begin{proof}
We prove it by contradiction again.
If $\mathbf{Q}_k=(\bar{Q}_1,\cdots,\bar{Q}_k) \in \partial \Lambda_k$, then there exists $(i,j)$ such that $|\bar{Q}_i-\bar{Q}_j|=\rho$.  Without loss of generality, we assume $(i,j)=(i,k)$. Then following the estimates in (\ref{eq601s}), (\ref{eq602s}), (\ref{eq603s}) and (\ref{eq604s}), we have

\begin{eqnarray}
\mathcal{C}_{k+1}&&=J(u_{\bar{Q}_1,\cdots,\bar{Q}_k},v_{\bar{Q}_1,\cdots,\bar{Q}_k})\\
&&\leq \mathcal{C}_{k-1}+I(U,V)+\frac{\delta}{2}\int_{\R^N} a U^2_{Q_k}+b V_{Q_k}dx\nonumber\\
&&-\frac{1}{4}A\gamma_1 \sum_{i=1}^{k-1}w(|Q_k-Q_i|)+O(e^{-\xi \rho}\sum_{i=1}^{k-1}w( |Q_k-Q_i|))+O(\delta)\nonumber\\
&&\leq \mathcal{C}_{k-1}+I(U,V)\nonumber\\
&&+C\delta-\frac{1}{4}A\gamma_1 \sum_{i=1}^{k-1}w(|Q_k-Q_i|)+O(e^{-\xi \rho}\sum_{i=1}^{k-1}w( |Q_k-Q_i|)).\nonumber
\end{eqnarray}

Similar to Section \ref{2.3}, by the definition of the configuration set, we have
\begin{eqnarray}
\mathcal{C}_{k+1}
&&\leq \mathcal{C}_{k-1}+I(U,V)+c\delta-\frac{1}{8}\gamma_1 w(\rho)+O(e^{-(1+\xi) \rho})\\
&&< \mathcal{C}_{k-1}+I(U,V).\nonumber
\end{eqnarray}

This is a contradiction with Lemma \ref{lemma601s}.
\end{proof}

\subsection{Proof of Theorem \ref{theo2}}\label{3.4}
\setcounter{equation}{0}
In this section, we apply the results in Section \ref{3.1}, \ref{3.2} and \ref{3.3} to prove Theorem \ref{theo2}.

\medskip

\noindent
{\it Proof of Theorem \ref{theo2}:} By Proposition \ref{p401s} in Section \ref{3.1}, there exists $\rho_0$ such that for $\rho>\rho_0$, we have $C^1$ map which, to any $\mathbf{Q}^\circ \in \Lambda_k$, associates $\phi_{\mathbf{Q}^\circ}$ such that
\begin{equation}
S(\vect{U_{\mathbf{Q}^\circ}+\phi_{\mathbf{Q}^\circ}}{V_{\mathbf{Q}^\circ}+\psi_{\mathbf{Q}^\circ}})=\sum_{i=1,\cdots,k,j=1,\cdots,n}c_{ij}\bar{Z}_{ij},\ \
\langle \vect{\phi_{\mathbf{Q}^\circ}}{\psi_{\mathbf{Q}^\circ}},\bar{Z}_{ij}\rangle=0,
\end{equation}
for some constants $\{c_{ij}\}\in \R^{kN}$.

From Proposition \ref{p602s} in Section \ref{3.2}, there is a $\mathbf{Q}\in \Lambda_k^\circ$ that achieves the maximum for the maximization problem in Proposition \ref{p602s}. Let $\vect{u_{\mathbf{Q}^\circ}}{v_{\mathbf{Q}^\circ}}=\vect{U_{\mathbf{Q}^\circ}}{V_{\mathbf{Q}^\circ}}
+\vect{\phi_{\mathbf{Q}^\circ}}{\psi_{\mathbf{Q}^\circ}}$. Then we have
\begin{equation}
D_{Q_{ij}}|_{Q_i=Q_i^\circ}\mathcal{M}(\mathbf{Q}^\circ)=0,\  \ i=1,\cdots,k, \ \ j=1,\cdots,N.
\end{equation}
Similar to the proof of Theorem \ref{theo1},
\begin{equation}\label{e601s}
\sum_{i=1,\cdots,k, \ j=1,\cdots,N}c_{ij}\int_{\R^n} \bar{Z}_{ij,1}\frac{\partial (U_{\mathbf{Q}}+\phi_{\mathbf{Q}})}{\partial Q_{sl}}|_{Q_s=Q_s^\circ}+\bar{Z}_{ij,2}\frac{\partial (V_{\mathbf{Q}}+\psi_{\mathbf{Q}})}{\partial Q_{sl}}|_{Q_s=Q_s^\circ}dx=0,
\end{equation}
for $s=1,\cdots,k, l=1,\cdots,N$.

We can get that (\ref{e601s}) is a diagonally dominant system for $c_{sl}$. So $c_{sl}=0$ for $s=1,\cdots,k, l=1,\cdots,N$. Hence $(u_{\mathbf{Q}^\circ},v_{\mathbf{Q}^\circ})$ is a solution of (\ref{p-2}).

\end{document}